\documentclass[journal]{IEEEtran}

\usepackage{amssymb}
\usepackage[pdftex]{graphicx}
\usepackage{amsmath}
\usepackage{bm}
\usepackage{cite}
\usepackage{color}
\usepackage{algorithmic}
\usepackage{algorithm}
\usepackage{mathtools}
\usepackage{framed}
\usepackage{tikz}
\usetikzlibrary{arrows.meta}

\mathtoolsset{showonlyrefs=true}

\usepackage{amsfonts}
\usepackage{enumerate}
\usepackage{array}
\usepackage{tabularx}

\newtheorem{theorem}{Theorem}
\newtheorem{assumption}{Assumption}

\newtheorem{remark}{Remark}
\newtheorem{lemma}{Lemma}

\newtheorem{corollary}{Corollary}
\newtheorem{condition}{Condition}
\newtheorem{proposition}{Proposition}

\newtheorem{proof}{Proof}

\newcommand{\argmin}{\mathop{\rm argmin}\limits}

\DeclareMathOperator{\rank}{rank}
\DeclareMathOperator{\diag}{diag}
\DeclareMathOperator{\tr}{tr}
\DeclareMathOperator{\Ker}{Ker}
\DeclareMathOperator{\dist}{dist}

\DeclareMathOperator{\lin}{lin}

\newcommand{\e}{\mathrm{e}}
\newcommand{\D}{\mathrm{d}}

\newcommand{\Real}{\mathbb{R}}

\def\qed{\hfill $\Box$}

\begin{document}

\title{Projected-Gradient Analysis for Open-Domain Convex Optimization under Boundary Blow-Up:
Application to Controllability Scoring}

\author{Kazuhiro Sato\thanks{K. Sato is with the Department of Mathematical Informatics, Graduate School of Information Science and Technology, The University of Tokyo, Tokyo 113-8656, Japan, email: kazuhiro@mist.i.u-tokyo.ac.jp}}

\maketitle

\begin{abstract}
We study convex optimization over a compact convex set when the objective
is smooth and convex only on an open domain.
Under a boundary-blow-up condition, every feasible initialization yields a
compact invariant sublevel set separated from the complement of the
objective domain, and an optimal solution exists.
For a domain-aware Armijo projected-gradient method, a safe-neighborhood
analysis establishes well-defined objective evaluations, finite
backtracking, sufficient decrease, and a run-specific but
iteration-independent positive lower bound on the accepted step sizes.
These properties yield explicit sublinear objective and stationarity
guarantees, together with convergence of the full iterate sequence.
Positive curvature restricted to feasible displacement directions further
guarantees uniqueness and linear convergence.
We apply the framework to controllability scoring with prescribed input
directions and compact convex allocation constraints.
Feasibility is characterized exactly by controllability of the input
directions eligible for positive allocation, while restricted injectivity
of the Gramian map guarantees uniqueness of the optimal allocation and
provides explicit strong-convexity bounds.
A directed-network example illustrates how candidate exclusion can preserve
or destroy feasibility and alter the optimal allocation in a
criterion- and horizon-dependent manner.
\end{abstract}


\IEEEpeerreviewmaketitle

\section{Introduction}
\label{sec:intro}

Projected-gradient and projection-arc Armijo methods are classical
first-order tools for smooth convex optimization over a constraint set
with an inexpensive Euclidean projection
\cite{bertsekas1976gradient,bertsekas2016nonlinear,
beck2017first,iusem2003projected}.
Their simple projection-based structure has also made projected-gradient
variants useful in control applications. Representative examples include
real-time and embedded model predictive control for linear systems with
input constraints
\cite{richter2012computational,patrinos2014accelerated},
nonlinear model predictive control based on projections onto linearized
constraints
\cite{torrisi2018projected},
and controllability maximization for large-scale network systems
\cite{sato2020controllability}.
A complication arises, however, when the objective is smooth only on an
open domain. Important matrix criteria, including the negative
log-determinant and inverse-trace criteria, are defined only for
positive-definite matrices. A trial step obtained by projecting onto a
separate explicit constraint set may therefore leave the objective domain
by reaching or crossing the boundary at which the underlying matrix becomes
singular.

The prior results most relevant to this open-domain difficulty provide
complementary ingredients.
Songsiri and Vandenberghe
\cite[Sec.~6.2]{songsiri2010topology}
introduced a domain-aware projection-arc method that rejects out-of-domain
trials before objective evaluation and assumes a closed and bounded initial
sublevel set, but did not derive an iteration-independent acceptance
threshold or the resulting rate and full-sequence convergence guarantees
for their displayed open-domain rules.
For objectives defined on the whole space, Iusem
\cite[Thm.~2]{iusem2003projected}
proved full-sequence convergence, while Aguiar et al.
\cite[Lemma~24 and Thms.~25--26]{aguiar2023inexact}
derived objective and stationarity bounds under a globally Lipschitz
continuous gradient.
In particular, the step-size argument in
\cite[Lemma~24]{aguiar2023inexact}
uses the reverse Armijo inequality at the preceding rejected trial.
In the present setting, that trial may instead lie outside the objective
domain, where the objective value and hence the reverse Armijo inequality
are unavailable.
Zhang and Hong
\cite{zhang2024firstorder}
developed a compact-set localization framework for locally smooth
objectives, but assumed that the objective remains defined on the regions
used by the algorithm and therefore did not address domain validity of
projected trials.
We connect these ingredients through a run-specific safe-neighborhood
analysis that guarantees both domain membership and Armijo acceptance for
all sufficiently small projected trials.

The open-domain algorithmic analysis developed in this paper provides the
analytical foundation for a generalization of the controllability-scoring
problems studied in
\cite{sato2024scores,sato2025uniqueness}.
Specifically, we extend the admissible actuation model from node-wise inputs
to arbitrary prescribed input directions and replace the full allocation
simplex by a compact convex allocation set.
This generalized setting also brings into focus a step-size property that
was left unresolved in the existing convergence analysis.
In the conventional node-wise setting,
\cite{sato2024scores} introduced an Armijo projected-gradient algorithm for
computing the volumetric controllability score (VCS) and average-energy
controllability score (AECS) and established convergence of the generated
sequence.
Reference~\cite[Thm.~6]{sato2025uniqueness} subsequently established linear
convergence for the finite-horizon scoring problems, provided that the
accepted step sizes are uniformly bounded away from zero.
However, this property is not derived from the Armijo line search in
\cite{sato2024scores} and is imposed as an assumption in
\cite[Thm.~6]{sato2025uniqueness}.
The present open-domain analysis derives it from the problem structure and
the domain-aware line-search mechanism.

The main contributions of this paper are as follows.
\begin{enumerate}
\item
We develop a quantitative safe-neighborhood analysis for the domain-aware
projection arc.
For every feasible initialization, the boundary-blow-up condition yields a
compact initial sublevel set separated from the complement of the objective
domain and hence a compact neighborhood on which the gradient is bounded
and Lipschitz continuous.
We show that every sufficiently small projected trial generated from this
sublevel set is both inside the objective domain and acceptable under the
Armijo condition.
This yields valid objective evaluations, well-defined and finitely
terminating backtracking, and an explicit run-specific but
iteration-independent positive lower bound on the accepted step sizes,
without invoking a reverse Armijo inequality at a potentially
out-of-domain rejected trial.

\item
Using this iteration-independent positive lower bound on the accepted step sizes, we
establish convergence and quantitative rate guarantees without assuming
that the gradient is globally Lipschitz continuous.
These include convergence of the full iterate sequence, an
$\mathcal O(1/K)$ objective-residual bound, and an
$\mathcal O(K^{-1/2})$ bound on the smallest stationarity residual over the
first $K+1$ outer iterations.
Under an additional curvature condition imposed only along directions
generated by differences between feasible points, we further obtain
uniqueness and an explicit linear-convergence estimate.

\item
We generalize the finite-horizon VCS and AECS problems from
node-wise actuation to an arbitrary finite family of prescribed input
directions and compact convex allocation sets.
Feasibility is characterized exactly by controllability of the input
directions eligible for positive allocation, while restricted injectivity
of the Gramian map guarantees uniqueness of the VCS and AECS and provides explicit strong-convexity bounds.
 For candidate-exclusion constraints, we further derive a
lower-dimensional simplex reformulation and an explicit projection formula,
characterize restricted injectivity through the full column rank and
minimum singular value of a reduced Gramian matrix, and obtain explicit
expressions for the quantities entering the strong-convexity bounds.
A directed-network example illustrates how candidate exclusion can preserve
or destroy feasibility and alter the optimal allocations in a
criterion- and horizon-dependent manner.
In the example, exclusion produces larger $\ell_1$ allocation
changes for AECS than for VCS at all considered horizons.
\end{enumerate}

The remainder of this paper is as follows.
Section~\ref{sec:general_problem} develops the general open-domain theory.
Section~\ref{sec:generalized_controllability_scoring} gives the detailed
controllability-scoring application. Section~\ref{sec:numerical_experiment}
presents  a directed-network numerical study, and
Section~\ref{sec:conclusion} concludes the paper.

{\it Notation:}
The sets of real and nonnegative real numbers are denoted by $\Real$ and
$\Real_+:=[0,\infty)$, respectively. All vectors are column vectors, and
$(\cdot)^\top$ denotes transpose. The symbols $I_n$, $O$, and $\bm 1_n$
denote the $n\times n$ identity matrix, a zero matrix of the appropriate
size, and the all-ones vector in $\Real^n$, respectively; subscripts are
omitted when their dimensions are clear.

For $x\in\Real^n$, $\|x\|$ denotes the Euclidean norm and $\|x\|_1$
denotes the $\ell_1$-norm. For a matrix $X$, $\|X\|_2$ and
$\|X\|_{\rm F}$ denote the spectral and Frobenius norms, respectively.
The trace, determinant, diagonal-matrix operator, and rank are denoted by
$\tr$, $\det$, $\diag$, and $\rank$, respectively. For a symmetric matrix
$X$, $\lambda_{\min}(X)$ and $\lambda_{\max}(X)$ denote its smallest and
largest eigenvalues, while $\sigma_{\min}(X)$ denotes the smallest singular
value of a matrix.

The space of real symmetric $n\times n$ matrices is denoted by
$\mathbb S^n$. For $X\in\mathbb S^n$, $X\succeq O$ and $X\succ O$ mean
that $X$ is positive semidefinite and positive definite, respectively. For
$X,Y\in\mathbb S^n$, we write $X\succeq Y$ if $X-Y\succeq O$. For
$X\succ O$, $X^{1/2}$ denotes its symmetric positive-definite square root,
and $X^{-1/2}:=(X^{1/2})^{-1}$.

For a nonempty subset $C$ of a finite-dimensional Euclidean space, $\lin C$ and $\operatorname{ri}C$ denote its linear span
and relative interior, respectively, and
$C-C:=\{p-q\mid p,q\in C\}$.
The
complement of a set $A$ in its ambient Euclidean space is denoted by
$A^c$. For a finite set $S$, $|S|$ denotes its cardinality. For a linear
map $\Phi$, $\Ker\Phi$ denotes its kernel.

For a nonempty closed convex set $C\subset\Real^n$, $P_C$ denotes the
Euclidean projection onto $C$. For two nonempty sets
$A,B\subset\Real^n$, define
\begin{align}
\dist(A,B)
:=
\inf\left\{
\|a-b\|
\ \middle|\
a\in A,\ b\in B
\right\}.
\end{align}
For $a\in\Real^n$, we abbreviate $\dist(\{a\},B)$ as $\dist(a,B)$.

For a differentiable function $f$, $\nabla f(p)$ denotes its gradient and
${\rm D} f(p)[v]$ denotes its directional derivative at $p$ along $v$. If $f$
is twice differentiable, $\nabla^2f(p)$ denotes its Hessian. Landau
asymptotic notation is written as $\mathcal O(\cdot)$ to distinguish it
from the zero matrix $O$.

\section{Open-domain analysis under boundary blow-up}
\label{sec:general_problem}

We consider the open-domain constrained optimization problem
\begin{framed}
\vspace{-1em}
\begin{align}
\label{prob:general}
    \begin{aligned}
        &&& \text{minimize} && f(p) \\
        &&& \text{subject to} && p\in C\cap U.
    \end{aligned}
\end{align}
\vspace{-1em}
\end{framed}
\noindent
Here, $C\subset\Real^n$ is the explicit constraint set used for projection,
$U\subset\Real^n$ is the open domain of the objective, and
$f:U\to\Real$ is the objective function. 
Projection will be performed onto $C$, while membership in $U$
will be checked separately; no projection onto $C\cap U$ will be required.

Before introducing the projected point and the gradient mapping, we state
the standing assumptions used throughout this section. The first assumption
specifies the geometry of the explicit constraint set and the open objective
domain. Compactness and convexity of $C$ ensure that the Euclidean projection
onto $C$ is well defined and provide the boundedness needed later for
attainment and localization, while convexity of $C$ and $U$ makes the
feasible set $C\cap U$ convex. The second assumption supplies the
differential and convexity properties required by the projected-gradient
and curvature analyses.

\begin{assumption}
\label{ass:set}
The set $C\subset\Real^n$ is nonempty, compact, and convex, the set
$U\subset\Real^n$ is open and convex, and $C\cap U\neq\emptyset$.
\end{assumption}

\begin{assumption}
\label{ass:objective_regularity}
The function $f:U\to\Real$ is twice continuously differentiable and convex.
\end{assumption}

\subsection{Relation to existing projected-gradient results} \label{subsec2-A}

\begin{table*}[t]
\centering
\caption{Comparison with closely related projected-gradient analyses.}
\label{tab:related_open_domain_results}
{\footnotesize
\setlength{\tabcolsep}{3.0pt}
\renewcommand{\arraystretch}{1.12}

\begin{tabularx}{\linewidth}{%
@{}
p{2.95cm}
>{\centering\arraybackslash}p{2.35cm}
>{\centering\arraybackslash}p{2.35cm}
X
@{}}
\multicolumn{4}{c}{%
\textnormal{(a) Open-domain treatment and localization premises}
}\\
\hline
Reference
&
\shortstack{Objective not defined\\on all of $\Real^n$}
&
\shortstack{Domain test before\\objective evaluation}
&
Localization or smoothness premise
\\
\hline

Songsiri and Vandenberghe
\cite[Sec.~6.2]{songsiri2010topology}
&
Yes
&
Yes
&
Closed and bounded initial sublevel set for a prescribed feasible
initialization.
\\

Iusem
\cite[Thm.~2]{iusem2003projected}
&
No
&
No
&
Objective defined on $\Real^n$; existence of an optimal solution.
\\

Aguiar et al.~
\cite[Lemma~24 and Thms.~25--26]{aguiar2023inexact}
&
No
&
No
&
Objective defined on $\Real^n$; globally Lipschitz continuous gradient
for the stated rate results.
\\

Zhang and Hong
\cite{zhang2024firstorder}
&
No
&
No
&
Objective defined on $\Real^n$; local Lipschitz continuity on compact
sets and sequentially constructed local regions.
\\

Present work
&
Yes
&
Yes
&
Boundary blow-up, yielding a compact and boundary-separated initial
sublevel set for every feasible initialization.
\\
\hline
\end{tabularx}

\vspace{1.2ex}

\begin{tabular}{@{}p{3.10cm}p{2.75cm}p{5.45cm}p{2.65cm}p{2.65cm}@{}}
\multicolumn{5}{c}{%
\textnormal{(b) Step-size, rate, and convergence guarantees}
}\\
\hline
Reference
&
\shortstack{Iteration-independent\\positive step-size bound}
&
Nonasymptotic guarantee
&
\shortstack{Full-sequence\\convergence}
&
\shortstack{Linear\\convergence}
\\
\hline

Songsiri and Vandenberghe
\cite[Sec.~6.2]{songsiri2010topology}
&
Not stated
&
Not stated
&
Not stated
&
Not stated
\\

Iusem
\cite[Thm.~2]{iusem2003projected}
&
Not stated
&
Not stated
&
Yes
&
Not stated
\\

Aguiar et al.~
\cite[Lemma~24 and Thms.~25--26]{aguiar2023inexact}
&
Yes
&
$\mathcal O(K^{-1/2})$ stationarity and
$\mathcal O(K^{-1})$ objective-residual bounds
&
Yes, under their assumptions
&
Not stated
\\

Zhang and Hong
\cite{zhang2024firstorder}
&
Not stated
&
Gradient-oracle complexity in terms of local smoothness growth
&
Not stated
&
Not stated
\\

Present work
&
Yes
&
$\mathcal O(K^{-1/2})$ stationarity and
$\mathcal O(K^{-1})$ objective-residual bounds
&
Yes
&
Yes, under feasible-direction curvature
\\
\hline
\end{tabular}
}
\end{table*}

Under Assumptions~\ref{ass:set} and
\ref{ass:objective_regularity}, the projected point and gradient mapping
below are well defined. To fix notation, for
$p\in C\cap U$ and $\alpha>0$, let $q_\alpha(p):=P_C(p-\alpha\nabla f(p))$ and
\begin{align}
\mathcal G_\alpha(p):=\alpha^{-1}(p-q_\alpha(p)).
\label{eq:projected_gradient_mapping}
\end{align}
Classical treatments of projected-gradient methods can be found in
\cite{bertsekas1976gradient},
\cite[Sec.~3.3]{bertsekas2016nonlinear}, and
\cite[Chap.~10]{beck2017first}.
The nonexpansiveness of the Euclidean projection
\cite[Thm.~6.42(b)]{beck2017first}, together with $P_C(p)=p$, gives
\begin{align}
\|q_\alpha(p)-p\|
&\leq
\alpha\|\nabla f(p)\|.
\label{eq:cited_trial_displacement}
\end{align}
Moreover, the variational characterization of the Euclidean projection
\cite[Thm.~6.41]{beck2017first}, applied with $p\in C$, yields
\begin{align}
\nabla f(p)^\top(q_\alpha(p)-p)
&\leq
-\alpha^{-1}\|q_\alpha(p)-p\|^2.
\label{eq:cited_projection_descent}
\end{align}
The projection characterization and the convexity of $f$ show that
$\mathcal G_\alpha(p)=0$ implies that $p$ solves
problem~\eqref{prob:general}.
The converse implication is standard when the objective is defined on the
entire constraint set; see
\cite[Cor.~10.8]{beck2017first}.
In the present setting, although $f$ is defined only on the open set $U$,
the openness of $U$ and the convexity of $C$ ensure that the converse
implication remains valid. Consequently,
\begin{align}
\mathcal G_\alpha(p)=0
\quad\Leftrightarrow\quad
p\ \text{is an optimal solution to \eqref{prob:general}}.
\label{eq:cited_gradient_mapping_optimality}
\end{align}

Table~\ref{tab:related_open_domain_results} compares the present analysis
with the projected-gradient results most directly relevant to it.
Panel~(a) summarizes their treatment of the objective domain, domain testing,
and localization or smoothness assumptions, while panel~(b) compares the
step-size, rate, and convergence guarantees explicitly established under
the respective assumptions.
In particular, the table distinguishes the run-specific sublevel-set
localization assumed for a prescribed initialization in
\cite[Sec.~6.2]{songsiri2010topology} from the problem-level
boundary-blow-up condition used here to obtain such localization for every
feasible initialization.
A ``Yes'' entry indicates that the corresponding property is explicitly
established for the method and assumptions treated in the cited reference;
``Not stated'' does not imply that the property is false.

\subsection{Boundary blow-up and initial-sublevel localization}
\label{subsec2-B}

For a fixed feasible initialization, the line-search analysis uses only
the geometry of the initial sublevel set associated with that run.
Fix $p^{(0)}\in C\cap U$ and define
\begin{align}
\mathcal F^{(0)}
:=
\left\{
p\in C\cap U
\ \middle|\
f(p)\leq f(p^{(0)})
\right\}.
\label{eq:initial_sublevel_compact}
\end{align}
More precisely, the analysis below requires
$\mathcal F^{(0)}$ to be compact and to remain a positive distance from
the complement of the objective domain. Rather than imposing these
properties separately for the chosen initial point, we derive them from
the following problem-level condition.

\begin{assumption}[Boundary blow-up]
\label{ass:blowup}
For $\alpha\in\Real$, define
\begin{align}
L_\alpha
:=
\left\{
p\in C\cap U
\ \middle|\
f(p)\leq\alpha
\right\}.
\end{align}
Whenever $L_\alpha$ is nonempty,
$\dist(L_\alpha,U^c)>0$.
\end{assumption}

The term boundary blow-up reflects the following equivalent sequential
form: for every sequence
$\{p^{(j)}\}\subset C\cap U$,
\begin{align}
\dist(p^{(j)},U^c)\to0
\quad\Rightarrow\quad
f(p^{(j)})\to+\infty.
\label{eq:sequential_boundary_blowup}
\end{align}
Indeed, a sequence approaching the complement of the objective domain
cannot remain in a fixed sublevel set. Thus, a sequence with uniformly bounded objective values cannot approach the complement of the objective domain.

Applying Assumption~\ref{ass:blowup} with
$\alpha=f(p^{(0)})$ gives the boundary separation required for the
initial sublevel set. Together with the compactness of $C$ and the
continuity and convexity of $f$, this also yields the remaining
geometric properties used below.

\begin{lemma}
\label{lem:sublevel_geometry}
Under Assumptions~\ref{ass:set}--\ref{ass:blowup},
the initial sublevel set $\mathcal F^{(0)}$ is nonempty, compact, and
convex, and
\begin{align}
d
:=
\dist\left(\mathcal F^{(0)},U^c\right)
>0.
\label{eq:sublevel_boundary_distance}
\end{align}
\end{lemma}

\begin{proof}
See Appendix~\ref{app:sublevel_geometry_and_attainment}.
\qed
\end{proof}

\begin{corollary}
\label{thm:existence}
Under Assumptions~\ref{ass:set}--\ref{ass:blowup},
problem~\eqref{prob:general} has an optimal solution.
Moreover, if $f$ is strictly convex on
$\mathcal F^{(0)}$, then the optimal solution is unique.
\end{corollary}

\begin{proof}
See Appendix~\ref{app:sublevel_geometry_and_attainment}.
\qed
\end{proof}

\begin{remark}
Assumption~\ref{ass:blowup} is imposed on the original objective and its
open domain relative to the explicit constraint set, rather than through
an auxiliary barrier introduced to encode $U$.
The condition holds automatically when $U=\Real^n$, under the convention
$\dist(A,\emptyset)=+\infty$.
\end{remark}

\subsection{Domain-aware projected gradient and line search} \label{subsec2-C}

Algorithm~\ref{alg:PG_general} applies projected-gradient steps to the
explicit constraint set $C$ while separately enforcing membership in the
open objective domain $U$. At outer iteration $k$, the trial step size is
initialized at $\bar\alpha$. The inner \texttt{while} loop, referred to
below as the backtracking loop, repeatedly forms the projected trial point
$q
:=
P_C\left(
p^{(k)}-\alpha\nabla f(p^{(k)})
\right)$.
Because the objective need not be defined outside $U$, the algorithm first
tests whether $q\in U$. Only after this domain test succeeds is $f(q)$
evaluated and the Armijo acceptance condition
\begin{align}
f(q)
\leq
f(p^{(k)})
+
\sigma\nabla f(p^{(k)})^\top
\left(q-p^{(k)}\right)
\label{eq:armijo_acceptance_condition}
\end{align}
checked. If the domain test fails, or if the Armijo condition fails after
the domain test succeeds, the trial is rejected and the step size is
reduced from $\alpha$ to $\rho\alpha$. Once both conditions are satisfied,
the trial point and step size are accepted by setting
$\tilde p^{(k)}:=q$ and $\alpha_k:=\alpha$, and the backtracking loop
terminates.

After acceptance, the algorithm tests stationarity before updating the
iterate. From \eqref{eq:projected_gradient_mapping},
the stopping condition is precisely
$\|\mathcal G_{\alpha_k}(p^{(k)})\|\leq\varepsilon$. Otherwise,
$\tilde p^{(k)}$ becomes the next iterate. The projection is taken onto $C$,
rather than $C\cap U$, because the latter need not be closed and its
Euclidean projection need not be well-defined. This separation between projection and domain membership is the
domain-aware mechanism described in
Subsection~\ref{subsec2-A}.

\begin{figure}[t]
\begin{algorithm}[H]
\caption{Domain-aware Armijo projected gradient method}
\label{alg:PG_general}
\textbf{Input:} $p^{(0)}\in C\cap U$, $\bar\alpha>0$,
$\sigma,\rho\in(0,1)$, and $\varepsilon\geq0$.
\begin{algorithmic}[1]
\FOR{$k=0,1,\ldots$}
    \STATE Initialize the backtracking loop with
    $\alpha:=\bar\alpha$.
    \WHILE{true} 
        \STATE Set
        $q:=P_C(p^{(k)}-\alpha\nabla f(p^{(k)}))$.
        \IF{$q\in U$}
            \IF{$f(q)\leq f(p^{(k)})
            +\sigma\nabla f(p^{(k)})^\top(q-p^{(k)})$}
                \STATE Set $\alpha_k:=\alpha$ and
                $\tilde p^{(k)}:=q$; exit the loop.
            \ENDIF
        \ENDIF
        \STATE Set $\alpha\leftarrow\rho\alpha$.
    \ENDWHILE
    \IF{$\|\tilde p^{(k)}-p^{(k)}\|/\alpha_k\leq\varepsilon$}
        \RETURN $p^{(k)}$.
    \ENDIF
    \STATE Set $p^{(k+1)}:=\tilde p^{(k)}$.
\ENDFOR
\end{algorithmic}
\end{algorithm}
\end{figure}

To establish finite backtracking uniformly over all outer iterations, we
introduce a common neighborhood on which domain membership and local
smoothness can be controlled. Define
\begin{align}
\mathcal K_r
:=
\left\{
q\in\Real^n
\ \middle|\
\dist(q,\mathcal F^{(0)})\leq r
\right\},
\qquad
r
:=
\min\left\{1,\frac{d}{2}\right\},
\label{eq:safe_neighborhood}
\end{align}
where $d$ is defined in Lemma~\ref{lem:sublevel_geometry}.
The bound $r\leq1$ is imposed only to keep $r$ finite in the whole-space
case $U=\Real^n$.
Then $\mathcal K_r$ is compact, convex, and contained in $U$, and
$\nabla f$ admits a Lipschitz constant on $\mathcal K_r$.
The proof of the next theorem shows that every sufficiently small
projected trial generated from $\mathcal F^{(0)}$, together with the
segment joining it to the current point, lies in $\mathcal K_r$.
The descent lemma~\cite[Lem.~5.7]{beck2017first} and the projection inequality
\eqref{eq:cited_projection_descent} then imply the Armijo condition
\eqref{eq:armijo_acceptance_condition}. Since
$\bar\alpha\rho^j\to0$ as $j\to\infty$, this yields finite backtracking
and an iteration-independent positive lower bound on the accepted step
sizes.

\begin{theorem}
\label{thm:well_defined}
Under Assumptions~\ref{ass:set}--\ref{ass:blowup}, every backtracking loop in
Algorithm~\ref{alg:PG_general} terminates after finitely many step-size
reductions.
Every accepted trial point $\tilde p^{(k)}$ belongs to
$\mathcal F^{(0)}$, and hence every generated iterate remains in
$\mathcal F^{(0)}$. The accepted step sizes satisfy
\begin{align}
\underline\alpha
\leq
\alpha_k
\leq
\bar\alpha,
\qquad
\underline\alpha
:=
\min\{\bar\alpha,\rho\alpha_\star\}
>0,
\label{eq:uniform_stepsize_lower_bound}
\end{align}
where
\begin{align}
\alpha_\star
&:=
\min\left\{
\frac{r}{G},
\frac{2(1-\sigma)}{L}
\right\},
\notag\\
G
&:=
\max\left\{
1,\max_{p\in\mathcal F^{(0)}}\|\nabla f(p)\|
\right\},
\notag\\
L
&:=
\max\left\{
1,\max_{q\in\mathcal K_r}\|\nabla^2 f(q)\|_2
\right\}.
\label{eq:uniform_acceptance_constants}
\end{align}
Here, $\mathcal K_r$ is defined in
\eqref{eq:safe_neighborhood}. Moreover, every accepted trial point satisfies
\begin{align}
f(\tilde p^{(k)})
\leq
f(p^{(k)})
-\sigma\alpha_k
\|\mathcal G_{\alpha_k}(p^{(k)})\|^2.
\label{eq:sufficient_decrease_mapping}
\end{align}
\end{theorem}
\begin{proof}
See Appendix~\ref{app:uniform_open_domain_safeguards}.
\qed
\end{proof}

The sets $\mathcal F^{(0)}$ and $\mathcal K_r$ serve complementary
purposes. In fact, $\mathcal F^{(0)}$ is invariant for accepted trial points and generated
iterates, whereas $\mathcal K_r$ controls projected trials before acceptance.
Because both sets are fixed for the entire run, the constants
$G$, $L$, and $\alpha_\star$ are independent of the outer iteration.
Accordingly, the uniform step-size bound
\eqref{eq:uniform_stepsize_lower_bound} and the sufficient-decrease estimate
\eqref{eq:sufficient_decrease_mapping} provide the two ingredients used in
the subsequent convergence analysis.

\begin{remark}
Following the standard extended-value convention
\cite[App.~A.3.3]{boyd2004convex}, define
\begin{align}
\bar f(p)
:=
\begin{cases}
f(p), & p\in C\cap U,\\
+\infty, & \text{otherwise}.
\end{cases}
\label{eq:extended_value_objective}
\end{align}
Under the compactness of $C$ and the continuity of $f$ on $U$,
Assumption~\ref{ass:blowup} is equivalent to closedness of $\bar f$.
This reformulation is useful for expressing the exclusion of
out-of-domain points and for establishing attainment, but it does not by
itself recover the standard Armijo analysis. Indeed, $\bar f$ is not a
finite-valued differentiable function outside $U$, and the segment joining
$p^{(k)}$ to an out-of-domain trial may contain points at which
$\nabla f$ is undefined. Moreover, representing membership in $U$ by an
indicator function would lead to a proximal subproblem over $C\cap U$,
whereas Algorithm~\ref{alg:PG_general} projects onto $C$ and then tests
membership in $U$ separately.
\end{remark}

\subsection{Localized rates and full-sequence convergence} \label{subsec2-D}

The next theorem derives objective and stationarity estimates using the
sufficient-decrease inequality
\eqref{eq:sufficient_decrease_mapping} and the uniform positive lower bound
on the accepted step sizes in
\eqref{eq:uniform_stepsize_lower_bound}. The latter is obtained from the
compact safe neighborhood
\eqref{eq:safe_neighborhood}, on which the gradient is bounded and
Lipschitz continuous.

\begin{theorem}
\label{thm:global_convergence}
Suppose that Assumptions~\ref{ass:set}--\ref{ass:blowup} hold and that
Algorithm~\ref{alg:PG_general} is run with $\varepsilon=0$. Let
$p^\star$ be any optimal solution and set $f^\star:=f(p^\star)$. Then, for
every integer $K\geq0$ such that
$p^{(0)},\ldots,p^{(K+1)}$ are generated,
\begin{align}
\min_{k=0,\ldots,K}
\|\mathcal G_{\alpha_k}(p^{(k)})\|
\leq
\sqrt{
\frac{f(p^{(0)})-f^\star}
{\sigma\underline\alpha(K+1)}
}.
\label{eq:global_stationarity_complexity}
\end{align}
Moreover,
\begin{align}
f(p^{(K)})-f^\star
\leq
\frac{
\|p^{(0)}-p^\star\|^2
+2\bar\alpha\bigl(f(p^{(0)})-f^\star\bigr)/\sigma
}{
2\underline\alpha(K+1)
}.
\label{eq:global_objective_complexity}
\end{align}
If the generated sequence is infinite, then
\begin{align}
\sum_{k=0}^{\infty}
\|\mathcal G_{\alpha_k}(p^{(k)})\|^2
&\leq
\frac{f(p^{(0)})-f^\star}
{\sigma\underline\alpha},
\label{eq:gradient_mapping_summability}\\
\|\mathcal G_{\alpha_k}(p^{(k)})\|
&\longrightarrow0
\qquad\text{as }k\to\infty.
\label{eq:gradient_mapping_vanishes}
\end{align}
\end{theorem}

\begin{proof}
See Appendix~\ref{app:localized_rates_and_full_sequence}.
\qed
\end{proof}

Theorem~\ref{thm:global_convergence} shows that
\eqref{eq:global_objective_complexity} yields an
$\mathcal O(1/K)$ objective-residual bound and that
\eqref{eq:global_stationarity_complexity} yields an
$\mathcal O(K^{-1/2})$ best-iterate stationarity bound.
The latter implies that, for any tolerance
$\varepsilon_{\rm tol}>0$, an iterate satisfying
$\|\mathcal G_{\alpha_k}(p^{(k)})\|
\leq
\varepsilon_{\rm tol}$
is obtained within
$\mathcal O(\varepsilon_{\rm tol}^{-2})$
accepted projected-gradient iterations.
In addition, \eqref{eq:gradient_mapping_vanishes} guarantees convergence
of the stationarity residual to zero along every infinite run.

Theorem~\ref{thm:well_defined} confines all iterates to the compact set
$\mathcal F^{(0)}$ and bounds the accepted step sizes away from zero,
whereas Theorem~\ref{thm:global_convergence} gives
\eqref{eq:gradient_mapping_vanishes}.
Consequently, every cluster point is optimal by continuity of the gradient
mapping and the optimality characterization
\eqref{eq:cited_gradient_mapping_optimality}.
A standard quasi-Fej\'er argument
\cite[Thm.~2]{iusem2003projected}
then yields convergence of the full iterate sequence.

\begin{corollary}
\label{cor:full_sequence_convergence}
Under Assumptions~\ref{ass:set}--\ref{ass:blowup},
Algorithm~\ref{alg:PG_general} with $\varepsilon=0$ either terminates at an
optimal solution or generates an infinite sequence that converges to an
optimal solution.
\end{corollary}

\begin{proof}
See Appendix~\ref{app:localized_rates_and_full_sequence}.
\qed
\end{proof}
\subsection{Feasible-direction curvature and linear convergence}
\label{sec:strong_convexity}

The preceding results require only convexity and therefore allow nonunique
optimal solutions; they also provide sublinear objective and stationarity
bounds. We now identify an additional curvature condition that guarantees
uniqueness and linear convergence. Requiring positive curvature in every
direction of $\Real^n$ would be unnecessarily strong because the feasible
set may lie in a lower-dimensional affine subspace. It is
therefore sufficient to control the curvature only along directions in which
two feasible points can differ.

To describe these directions, define
\begin{align}
\mathcal L_C
:=
\lin(C-C),
\label{eq:direction_space}
\end{align}
the linear space generated by all displacements between points of $C$.
Hence, every feasible displacement $q-p$, with $p,q\in C$, belongs to
$\mathcal L_C$. For
$p\in U$ and a unit vector $v\in\Real^n$, we call
\begin{align}
v^\top\nabla^2 f(p)v
=
\left.
\frac{\D^2}{\D t^2}f(p+tv)
\right|_{t=0}
\end{align}
the directional curvature of $f$ at $p$ along $v$. It measures the
second-order growth of $f$ relative to its affine first-order model along
that direction. Since objective comparisons between feasible points involve
only displacements in $\mathcal L_C$, curvature in directions orthogonal to
$\mathcal L_C$ is irrelevant to strong convexity over the feasible set.

The next result shows that pointwise positive directional curvature on
the compact invariant sublevel set yields a uniform strong-convexity
bound along the feasible-direction space $\mathcal L_C$. This is the
additional ingredient needed for uniqueness and linear convergence.

\begin{lemma}
\label{lem:compact_strong_convexity}
Suppose that $\mathcal L_C\ne\{0\}$ and that
\begin{align}
v^\top\nabla^2 f(p)v>0
\label{eq:PD_tangent}
\end{align}
for every $p\in\mathcal F^{(0)}$ and every
$v\in\mathcal L_C$ with $\|v\|=1$.
Define the minimum feasible-direction curvature by
\begin{align}
\mu
:=
\min_{\substack{
p\in\mathcal F^{(0)}\\
v\in\mathcal L_C,\ \|v\|=1}}
v^\top\nabla^2 f(p)v.
\label{eq:general_strong_convexity_constant}
\end{align}
Then $\mu>0$, and
\begin{align}
v^\top\nabla^2 f(p)v
\geq
\mu\|v\|^2
\label{eq:uniform_feasible_direction_curvature}
\end{align}
for every $p\in\mathcal F^{(0)}$ and every
$v\in\mathcal L_C$.

Consequently, $f$ is $\mu$-strongly convex on
$\mathcal F^{(0)}$; that is,
\begin{align}
f(q)
\geq
f(p)+\nabla f(p)^\top(q-p)
+\frac{\mu}{2}\|q-p\|^2
\label{eq:general_strong_convexity}
\end{align}
for every $p,q\in\mathcal F^{(0)}$.
\end{lemma}
\begin{proof}
See Appendix~\ref{app:feasible_direction_curvature_and_linear_convergence}.
\qed
\end{proof}

Linear convergence of projected gradient for a globally smooth and
strongly convex objective is classical
\cite{bertsekas1976gradient,beck2017first}. Here, however, neither
smoothness nor strong convexity is imposed globally: both are obtained only
on the invariant initial sublevel set $\mathcal{F}^{(0)}$, and positive curvature is required
only along the feasible direction space $\mathcal{L}_C$. The proof in
Appendix~\ref{app:feasible_direction_curvature_and_linear_convergence}
localizes the classical argument to this open-domain setting and derives the
explicit contraction factor for the variable step sizes generated by the
domain-aware Armijo search.

\begin{theorem}
\label{thm:linear_convergence}
Suppose Assumptions~\ref{ass:set}--\ref{ass:blowup} and
\eqref{eq:PD_tangent} hold. Then the optimum $p^\star$ is unique. For
$\mu$ in Lemma~\ref{lem:compact_strong_convexity} when
$\mathcal L_C\ne\{0\}$, any sequence $\{p^{(k)}\}$ generated
by Algorithm~\ref{alg:PG_general} before termination satisfies
\begin{align}
\|p^{(k)}-p^\star\|
\leq
\sqrt{\frac{2(f(p^{(0)})-f(p^\star))}{\mu}}\,\gamma^{k/2},
\label{eq:linear_iterate_convergence}
\end{align}
where
\begin{align}
\gamma:=1-\sigma\min\{1,\mu\underline\alpha\}\in(0,1).
\label{eq:linear_rate}
\end{align}
\end{theorem}
\begin{proof}
See Appendix~\ref{app:feasible_direction_curvature_and_linear_convergence}.
\qed
\end{proof}

\section{Controllability Scoring with Prescribed Input Directions}
\label{sec:generalized_controllability_scoring}

This section generalizes the controllability-scoring problems in
\cite{sato2024scores,sato2025uniqueness} by allocating a limited actuation
resource among an arbitrary finite collection of prescribed input
directions, possibly fewer than the state dimension, over a nonempty compact
convex subset of the allocation simplex.
Unlike target controllability scoring
\cite{sato2025target}, which restricts the controlled outputs through an
output controllability Gramian, the present formulation retains full-state
controllability and generalizes the admissible input directions and their
allocations.

We apply the open-domain optimization framework developed in
Section~\ref{sec:general_problem} to establish existence, computation, and
uniqueness properties for this generalized problem. This application
requires four ingredients not supplied directly by the abstract framework:
identification of the positive-definiteness domain, verification of boundary
blow-up for the VCS and AECS objectives, characterization of feasible
initialization under the restricted input set, and translation of
feasible-direction Hessian curvature into system-theoretic conditions.
Although the number of prescribed input directions may be smaller than the
state dimension, full-state controllability remains necessary for the
finite-horizon controllability Gramian to be positive definite.

Consider an $n$-dimensional linear system with autonomous dynamics
\begin{align}
\dot{x}(t)
=
Ax(t),
\qquad
A\in\Real^{n\times n}.
\label{eq:generalized_CS_autonomous_system}
\end{align}
We prescribe an arbitrary collection of $m$ candidate input directions,
\begin{align}
b_1,\ldots,b_m\in\Real^n,
\qquad
m\geq1.
\label{eq:generalized_candidate_directions}
\end{align}
The number and geometry of these directions are prescribed by the available
actuation mechanism and need not coincide with the state dimension or the
coordinate directions. In particular, no relation between $m$ and $n$ is
imposed: the number of candidate input directions may be smaller than, equal
to, or larger than the state dimension, and the vectors
$b_1,\ldots,b_m$ need not be linearly independent. Thus, the formulation
covers restricted collections with $m<n$ as well as possibly redundant
collections with $m>n$.

Following the continuous resource-allocation formulation used in
controllability scoring
\cite{sato2024scores,sato2025uniqueness}, we represent the distribution of
actuation resource among the prescribed candidate directions by
\begin{align}
p
=
\begin{bmatrix}
p_1 & \cdots & p_m
\end{bmatrix}^{\top}
\in\Real^m_+.
\end{align}
The allocation \(p_i\) scales the \(i\)th candidate direction \(b_i\)
through the resource-weighted input matrix
\begin{align}
B(p)
:=
\begin{bmatrix}
\sqrt{p_1}b_1 & \cdots & \sqrt{p_m}b_m
\end{bmatrix},
\label{eq:resource_weighted_input_matrix}
\end{align}
so that the controlled system is
\begin{align}
\dot{x}(t)
=
Ax(t)+B(p)u(t),
\qquad
u(t)\in\Real^m.
\label{eq:generalized_controlled_system}
\end{align}
Accordingly, \(p_i\) quantifies the amount of actuation resource assigned
to the \(i\)th candidate direction. The interpretation of equal allocation
values depends on the scaling of the vectors \(b_i\): if the components of
\(p\) are to represent directly comparable resource levels, the candidate
directions should be normalized consistently, whereas fixed differences in
nominal actuator gains may instead be incorporated into the vectors \(b_i\).

We normalize the total actuation resource to one and define the allocation
simplex
\begin{align}
\Delta_m
:=
\left\{
p\in\Real^m
\ \middle|\
p_i\geq0,\quad
\sum_{i=1}^m p_i=1
\right\}.
\label{eq:generalized_allocation_simplex}
\end{align}
The choice $p\in\Delta_m$ represents an unrestricted allocation of the
normalized resource among all candidate input directions.
Additional restrictions on the admissible allocation are represented by a
nonempty compact convex set
\begin{align}
C\subseteq\Delta_m,
\label{eq:generalized_allocation_constraint}
\end{align}
and we assume that the Euclidean projection onto $C$ can be computed
efficiently. The set $C$ may impose individual upper or lower bounds,
groupwise budgets, or symmetry constraints. It also represents candidate
exclusion: prescribing $p_i=0$ removes the $i$th input direction from every
admissible allocation. 

In applying the general framework developed in
Section~\ref{sec:general_problem} to controllability scoring, we address
three questions: when the scoring problem is feasible, whether the
domain-aware algorithm converges without assuming uniqueness, and under
what conditions the optimal allocation is unique and the convergence is
linear.
\subsection{Generalized finite-horizon controllability scores}
\label{sec:generalized_CS_problem}

The key structural property of the generalized formulation is that the
finite-horizon controllability Gramian depends linearly on the resource
allocation. Fix a time horizon $T>0$. The controllability Gramian of
system~\eqref{eq:generalized_controlled_system} is
\begin{align}
W(p,T)
:=
\int_0^T
\e^{At}B(p)B(p)^\top\e^{A^\top t}
\,\D t.
\label{eq:generalized_finite_horizon_Gramian}
\end{align}
Since
$B(p)B(p)^\top
=
\sum_{i=1}^m p_i b_i b_i^\top$,
the Gramian admits the decomposition
\begin{align}
W(p,T)
=
\sum_{i=1}^m p_iW_i(T),
\label{eq:generalized_Gramian_decomposition}
\end{align}
where
\begin{align}
W_i(T)
:=
\int_0^T
\e^{At}b_i b_i^\top\e^{A^\top t}
\,\D t
\succeq O.
\label{eq:generalized_individual_Gramian}
\end{align}
Thus, $W_i(T)$ is the finite-horizon Gramian contribution of the $i$th
candidate input direction, while $p_i$ determines the weight assigned to
that contribution.

This linear decomposition identifies the natural open domain of the
controllability-score objectives. Define the linear Gramian map
\begin{align}
\Phi_T:\Real^m\to\mathbb S^n,
\qquad
\Phi_T(v)
:=
\sum_{i=1}^m v_iW_i(T).
\label{eq:generalized_Gramian_map}
\end{align}
Then $W(p,T)=\Phi_T(p)$, and the allocations for which the Gramian is
positive definite form the set
\begin{align}
X_T
:=
\left\{
p\in\Real^m
\ \middle|\
W(p,T)\succ O
\right\}.
\label{eq:generalized_positive_definiteness_domain}
\end{align}
Because $X_T$ is the inverse image of the positive-definite cone under the
linear map $\Phi_T$, it is an open convex subset of $\Real^m$. 

The generalized VCS and AECS problems are obtained by applying the standard
log-determinant and inverse-trace criteria to this linear Gramian
\cite{sato2024scores,sato2025uniqueness}. Specifically, for \(p\in X_T\),
define
\begin{align}
f_{\rm V}(p)
&:=
-\log\det W(p,T),
\label{eq:generalized_VCS_objective}\\
f_{\rm A}(p)
&:=
\tr\left(W(p,T)^{-1}\right).
\label{eq:generalized_AECS_objective}
\end{align}
For either \(f\in\{f_{\rm V},f_{\rm A}\}\), the generalized
controllability-scoring problem is
\begin{framed}
\vspace{-1em}
\begin{align}
\label{prob:generalized_controllability_scoring}
\begin{aligned}
&&& \text{minimize}   && f(p) \\
&&& \text{subject to} && p\in C\cap X_T.
\end{aligned}
\end{align}
\vspace{-1em}
\end{framed}
\noindent
The conventional node-wise controllability-scoring problem studied in
\cite{sato2024scores,sato2025uniqueness} is recovered by setting
\begin{align}
m=n,
\qquad
b_i=e_i
\quad
(i=1,\ldots,n),
\qquad
C=\Delta_n,
\label{eq:conventional_scoring_specialization}
\end{align}
where \(e_i\) denotes the \(i\)th standard basis vector of \(\Real^n\).

Problem~\eqref{prob:generalized_controllability_scoring}
has the structure studied in Section~\ref{sec:general_problem}: \(C\) is
the compact convex set of explicit allocation constraints, whereas \(X_T\)
is the open domain on which the objective is well defined.

We next characterize when the admissible set $C$ contains an allocation
with a positive-definite controllability Gramian, or equivalently, when
$C\cap X_T$ is nonempty. For a nonempty compact convex set
$C\subseteq\Delta_m$, define the set of candidate input directions that can
receive positive allocation by
\begin{align}
S_C
:=
\left\{
i\in\{1,\ldots,m\}
\ \middle|\
p_i>0\text{ for some }p\in C
\right\}.
\label{eq:active_candidate_set}
\end{align}
Let $B_{S_C}$ denote the matrix whose columns are the candidate directions
$b_i$ with $i\in S_C$.

\begin{lemma}
\label{lem:general_convex_feasibility}
Let $p^\circ\in\operatorname{ri}C$. For every $T>0$,
\begin{align}
C\cap X_T\neq\emptyset
&\quad\Leftrightarrow\quad
\sum_{i\in S_C}W_i(T)\succ O
\notag\\
&\quad\Leftrightarrow\quad
(A,B_{S_C})\text{ is controllable}.
\label{eq:general_convex_feasibility_equivalence}
\end{align}
Whenever these conditions hold, $p^\circ\in C\cap X_T$.
\end{lemma}

\begin{proof}
See Appendix~\ref{app:feasibility_and_candidate_exclusion}.
\qed
\end{proof}

Lemma~\ref{lem:general_convex_feasibility} shows that feasibility depends
on the allocation constraint $C$ only through $S_C$. Specifically, the
scoring problem~\eqref{prob:generalized_controllability_scoring} is feasible
if and only if $(A,B_{S_C})$ is controllable. When this condition holds,
every point in $\operatorname{ri}C$ belongs to $C\cap X_T$ and is therefore
a feasible initialization for Algorithm~\ref{alg:PG_general}. Otherwise,
$W(p,T)$ is singular for every $p\in C$, and hence
$C\cap X_T=\emptyset$.

\subsection{Differential formulas, convexity, and Hessian structure}
\label{sec:generalized_CS_differential_properties}

We collect the differential identities needed for the
projected-gradient and feasible-direction curvature analyses.
They follow from standard matrix differential formulas applied to the
linear Gramian map $W(p,T)=\Phi_T(p)$ and reduce to the formulas
established in the conventional node-wise controllability-scoring
setting \cite{sato2024scores,sato2025uniqueness}.
A short derivation is included in
Appendix~\ref{app:controllability_scoring_global_convergence}
for completeness, in particular to establish the Hessian-kernel
characterization used below.

\begin{lemma}
\label{lem:generalized_CS_differential_formulas}
Let $p\in X_T$ and $v\in\Real^m$, and set
\begin{align}
W:=W(p,T),
\qquad
H:=\Phi_T(v).
\end{align}
The first directional derivatives are
\begin{align}
{\rm D} f_{\rm V}(p)[v]
&=
-\tr\left(W^{-1}H\right),
\label{eq:generalized_VCS_first_derivative}\\
{\rm D} f_{\rm A}(p)[v]
&=
-\tr\left(W^{-1}HW^{-1}\right).
\label{eq:generalized_AECS_first_derivative}
\end{align}
The corresponding second directional derivatives are
\begin{align}
v^\top\nabla^2 f_{\rm V}(p)v
&=
\left\|
W^{-1/2}HW^{-1/2}
\right\|_{\rm F}^2,
\label{eq:generalized_VCS_second_derivative}\\
v^\top\nabla^2 f_{\rm A}(p)v
&=
2\tr\left(
\left(W^{-1}H\right)^2W^{-1}
\right).
\label{eq:generalized_AECS_second_derivative}
\end{align}
 Moreover, for either $f\in\{f_{\rm V},f_{\rm A}\}$,
\begin{align}
v^\top\nabla^2 f(p)v=0
\quad\Leftrightarrow\quad
\Phi_T(v)=O.
\label{eq:generalized_Hessian_kernel}
\end{align}
\end{lemma}
\begin{proof}
See Appendix~\ref{app:controllability_scoring_global_convergence}.
\qed
\end{proof}

Lemma~\ref{lem:generalized_CS_differential_formulas} implies that
the objective functions $f_V$ and $f_A$ are infinitely differentiable
and convex on $X_T$.
The convexity conclusion follows from the nonnegativity of the
second directional derivatives. 
The kernel characterization in
\eqref{eq:generalized_Hessian_kernel} will be used below to relate
positive curvature along feasible allocation directions to restricted
injectivity of the Gramian map.
Taking $v=e_i^{(m)}$, where $e_i^{(m)}$ denotes the $i$th standard
basis vector of $\Real^m$, gives
$\Phi_T(e_i^{(m)})=W_i(T)$. Hence, the gradient components required by
the projected-gradient method are
\begin{align}
\left(\nabla f_{\rm V}(p)\right)_i
&=
-\tr
\left(
W(p,T)^{-1}W_i(T)
\right),
\label{eq:generalized_VCS_gradient}\\
\left(\nabla f_{\rm A}(p)\right)_i
&=
-\tr
\left(
W(p,T)^{-1}
W_i(T)
W(p,T)^{-1}
\right).
\label{eq:generalized_AECS_gradient}
\end{align}

\subsection{Boundary blow-up and global convergence}
\label{sec:generalized_CS_boundary_blowup}

For the two matrix objectives, boundary blow-up follows from the spectral
behavior of the controllability Gramian. In particular, approaching the
boundary of $X_T$ forces the smallest eigenvalue of $W(p,T)$ to vanish,
while compactness of $C$ prevents its remaining eigenvalues from becoming
unbounded.

\begin{lemma}
\label{lem:generalized_CS_boundary_blowup}
Let $C\subseteq\Delta_m$ be nonempty, compact, and convex, and suppose that
$C\cap X_T\neq\emptyset$.
Then $f_{\rm V}$ and $f_{\rm A}$ satisfy
Assumption~\ref{ass:blowup} with $U=X_T$.
\end{lemma}

\begin{proof}
See Appendix~\ref{app:controllability_scoring_global_convergence}.
\qed
\end{proof}

The preceding results allow the general open-domain theory to be applied to
the generalized controllability-scoring problem.

\begin{theorem}[Existence and global convergence]
\label{thm:generalized_CS_global_convergence}
Let $C\subseteq\Delta_m$ be nonempty, compact, and convex, and suppose that
$C\cap X_T\neq\emptyset$. For either
$f\in\{f_{\rm V},f_{\rm A}\}$, problem
\eqref{prob:generalized_controllability_scoring} has an optimal solution.

Moreover, Algorithm~\ref{alg:PG_general} is well-defined for every
initialization $p^{(0)}\in C\cap X_T$, and all accepted iterates remain in
the compact initial sublevel set contained in $C\cap X_T$. With
$\varepsilon=0$, the algorithm either terminates at an optimal solution or
generates an infinite sequence that converges to an optimal solution,
without requiring uniqueness. In the infinite case, the best-iterate stationarity estimate \eqref{eq:global_stationarity_complexity}, 
 the objective-residual estimate
\eqref{eq:global_objective_complexity}, and the whole-sequence residual convergence
\eqref{eq:gradient_mapping_vanishes} hold with $U=X_T$.
\end{theorem}

\begin{proof}
See Appendix~\ref{app:controllability_scoring_global_convergence}.
\qed
\end{proof}

\subsection{Restricted injectivity and linear convergence}
\label{sec:generalized_CS_restricted_injectivity}

For the curvature-based uniqueness and convergence analysis, only directions
in
$\mathcal L_C=\lin(C-C)$ are relevant, because every displacement between two feasible
allocations belongs to $\mathcal L_C$.
The Hessian kernel characterization
\eqref{eq:generalized_Hessian_kernel} shows that positive curvature along
every nonzero direction in $\mathcal L_C$ is equivalent to injectivity of
the Gramian map on this direction space. This motivates the following
condition.

\begin{condition}
\label{cond:generalized_CS_restricted_injectivity}
The Gramian map $\Phi_T$ is injective on $\mathcal L_C$:
\begin{align}
\Ker\Phi_T\cap\mathcal L_C=\{0\}.
\label{eq:generalized_restricted_injectivity_condition}
\end{align}
Equivalently, $\Phi_T(v)\neq O$ for every nonzero
$v\in\mathcal L_C$.
\end{condition}

This condition is weaker than injectivity of $\Phi_T$ on $\Real^m$ because
only differences between feasible allocations need to be distinguished.
It gives pointwise positive curvature along $\mathcal L_C$; the next lemma
quantifies this curvature by explicit uniform lower bounds.

\begin{lemma}
\label{lem:generalized_CS_uniform_curvature}
Suppose that
Condition~\ref{cond:generalized_CS_restricted_injectivity} holds and that
$C\cap X_T\neq\emptyset$ and $\mathcal L_C\neq\{0\}$. Define
\begin{align}
\kappa_{C,T}
&:=
\min_{\substack{v\in\mathcal L_C\\ \|v\|=1}}
\|\Phi_T(v)\|_{\rm F}, \label{eq:explicit_curvature_parameters1}
\\
\beta_T
&:=
\max_{p\in C}
\lambda_{\max}\bigl(W(p,T)\bigr).
\label{eq:explicit_curvature_parameters2}
\end{align}
Then $\kappa_{C,T}>0$, and, for every $p\in C\cap X_T$ and
$v\in\mathcal L_C$,
\begin{align}
v^\top\nabla^2 f_{\rm V}(p)v
&\geq
\mu_{\rm V}\|v\|^2,
&
\mu_{\rm V}
&:=
\frac{\kappa_{C,T}^2}{\beta_T^2},
\label{eq:explicit_mu_vcs}\\
v^\top\nabla^2 f_{\rm A}(p)v
&\geq
\mu_{\rm A}\|v\|^2,
&
\mu_{\rm A}
&:=
\frac{2\kappa_{C,T}^2}{\beta_T^3}.
\label{eq:explicit_mu_aecs}
\end{align}
Consequently, $f_{\rm V}$ and $f_{\rm A}$ are strongly convex on
$C\cap X_T$ along the feasible-direction space $\mathcal L_C$, with
strong-convexity parameters $\mu_{\rm V}$ and $\mu_{\rm A}$,
respectively.
\end{lemma}

\begin{proof}
See Appendix~\ref{app:restricted_injectivity_and_linear_convergence}.
\qed
\end{proof}

Lemma~\ref{lem:generalized_CS_uniform_curvature} supplies the uniform strong-convexity condition required
by Theorem~\ref{thm:linear_convergence}. It therefore yields both uniqueness
of the controllability scores and linear convergence of the
domain-aware projected-gradient method.

\begin{theorem}
\label{thm:generalized_CS_linear_convergence}
Let the assumptions of
Theorem~\ref{thm:generalized_CS_global_convergence} hold, and suppose that
Condition~\ref{cond:generalized_CS_restricted_injectivity} is satisfied.
For either $f\in\{f_{\rm V},f_{\rm A}\}$, problem
\eqref{prob:generalized_controllability_scoring} has a unique optimal
solution $p^\star$.

If $\mathcal L_C\neq\{0\}$, then, from every
$p^{(0)}\in C\cap X_T$, Algorithm~\ref{alg:PG_general} with
$\varepsilon=0$ either terminates at $p^\star$ or converges linearly
to $p^\star$. The estimates in Theorem~\ref{thm:linear_convergence} hold
with
\begin{align}
\mu
=
\begin{cases}
\mu_{\rm V} &\text{if}\quad f=f_{\rm V},\\
\mu_{\rm A} &\text{if}\quad  f=f_{\rm A}.
\end{cases}
\end{align}
\end{theorem}

\begin{proof}
See Appendix~\ref{app:restricted_injectivity_and_linear_convergence}.
\qed
\end{proof}

\begin{remark}
    For the full simplex $C=\Delta_m$, the feasible-direction space is
$\mathcal L_{\Delta_m}
=
\left\{
v\in\Real^m
\ \middle|\
\bm 1^\top v=0
\right\}$.
Hence, Condition~\ref{cond:generalized_CS_restricted_injectivity} means that
\begin{align}
\sum_{i=1}^m v_iW_i(T)=O\quad \text{and}
\quad
\sum_{i=1}^m v_i=0
\label{eq:affine_dependence_gramians}
\end{align}
can hold only when $v=0$. This is precisely the definition of affine
independence of the matrices $W_1(T),\ldots,W_m(T)$, regarded as points in
the vector space of symmetric matrices.
To see the equivalent difference-matrix characterization, fix an arbitrary
index $r\in\{1,\ldots,m\}$. Since
$v_r=-\sum_{i\ne r}v_i$
for every $v\in\mathcal L_{\Delta_m}$, we have
\begin{align}
\sum_{i=1}^m v_iW_i(T)
=
\sum_{i\ne r}v_i\bigl(W_i(T)-W_r(T)\bigr).
\label{eq:affine_independence_difference_form}
\end{align}
Therefore, restricted injectivity on $\mathcal L_{\Delta_m}$ holds if and
only if the $m-1$ difference matrices
\begin{align}
W_i(T)-W_r(T),
\qquad i\ne r,
\end{align}
are linearly independent. This condition is weaker than linear independence
of $W_1(T),\ldots,W_m(T)$ because it excludes only linear dependencies whose
coefficients sum to zero.  
\end{remark}
   
\begin{remark}
  Uniqueness has already been studied for the conventional node-wise setting
$m=n$, $b_i=e_i$, and $C=\Delta_n$. When $A$ is Hurwitz,
\cite[Thm.~4]{sato2024scores} establishes finite-horizon uniqueness of the
VCS and AECS for every $T>0$. For an arbitrary system matrix,
\cite[Thm.~1]{sato2025uniqueness} proves uniqueness for almost every
$T>0$, with the exceptional horizons forming a set of Lebesgue measure zero.
The proofs of these results establish linear independence of the node-wise
Gramian contributions, or equivalently full injectivity of $\Phi_T$. This is
stronger than Condition~\ref{cond:generalized_CS_restricted_injectivity} required on the simplex.  
\end{remark}

\subsection{Candidate exclusion: feasibility and restricted injectivity}
\label{sec:generalized_CS_candidate_exclusion}

The preceding subsections treated an arbitrary nonempty compact convex
allocation set \(C\subseteq\Delta_m\). We now specialize the general
feasibility and restricted-injectivity results to candidate-exclusion
constraints, for which the admissible allocation set is a face of the
simplex. We first characterize whether the retained input directions
preserve feasibility. We then exploit the simplex-face structure to
reformulate the problem over a lower-dimensional simplex and derive an
explicit projection formula. Finally, we characterize Condition~\ref{cond:generalized_CS_restricted_injectivity} by the full column rank of
a reduced Gramian matrix and derive explicit expressions for the two
quantities \(\kappa_{C_S,T}\) and \(\beta_T\) entering the
strong-convexity bounds in Lemma~\ref{lem:generalized_CS_uniform_curvature}.

Candidate exclusion changes both the input family available for actuation
and the feasible allocation-direction space. The former may preserve or
destroy controllability and hence feasibility. The latter may remove
directions along which the Gramian map is unchanged and can therefore
affect restricted injectivity.

Let \(S\subseteq\{1,\ldots,m\}\) be a nonempty set of retained
candidate-channel indices, set \(k:=|S|\), and fix an ordering
\begin{align}
S=\{i_1,\ldots,i_k\}.
\label{ordering}
\end{align}
Channels outside \(S\) are excluded by fixing their allocations to zero,
so the admissible allocation set is the simplex face
\begin{align}
C_S
:=
\left\{
p\in\Delta_m
\ \middle|\
p_i=0
\text{ for every }i\notin S
\right\}.
\label{eq:candidate_exclusion_face}
\end{align}
The corresponding retained input directions are collected in
\begin{align}
B_S
:=
\begin{bmatrix}
b_{i_1} & \cdots & b_{i_k}
\end{bmatrix}
\in\Real^{n\times k}.
\label{eq:restricted_input_matrix}
\end{align}

\begin{corollary}
\label{cor:candidate_exclusion_feasibility}
Define the uniform allocation on \(C_S\) by
\begin{align}
p_i^{\rm uni}
:=
\begin{cases}
1/k, & i\in S,\\
0,   & i\notin S.
\end{cases}
\label{eq:uniform_allocation_on_face}
\end{align}
For every \(T>0\),
\begin{align}
C_S\cap X_T\neq\emptyset
&\quad\Leftrightarrow\quad
p^{\rm uni}\in X_T
\notag\\
&\quad\Leftrightarrow\quad
\sum_{i\in S}W_i(T)\succ O
\notag\\
&\quad\Leftrightarrow\quad
(A,B_S)\text{ is controllable}.
\label{eq:candidate_exclusion_feasibility_equivalence}
\end{align}
\end{corollary}

\begin{proof}
See Appendix~\ref{app:feasibility_and_candidate_exclusion}. \qed
\end{proof}

Corollary~\ref{cor:candidate_exclusion_feasibility} characterizes the first
effect of candidate exclusion. The reduced problem is feasible precisely
when \((A,B_S)\) is controllable. Equivalently, feasibility can be verified
by testing whether the uniform allocation \(p^{\rm uni}\) belongs to
\(X_T\). Since controllability of \((A,B_S)\) does not depend on the finite
horizon, the reduced problem is either feasible for every \(T>0\) or
infeasible for every \(T>0\).

We next consider the second effect, namely the change in Condition~\ref{cond:generalized_CS_restricted_injectivity}. For the simplex face \(C_S\), the
feasible-direction space is
\begin{align}
\mathcal L_{C_S}
=
\left\{
v\in\Real^m
\ \middle|\
v_i=0 \text{ for every } i\notin S,\quad
\sum_{i\in S}v_i=0
\right\}.
\label{eq:candidate_exclusion_direction_space}
\end{align}
Thus,
$\mathcal L_{C_S}
\subseteq
\mathcal L_{\Delta_m}$,
and \(\mathcal L_{C_S}\) contains only feasible perturbations supported on
the retained channels. Consequently, candidate exclusion may remove from
the feasible-direction space nonzero directions in
\(\Ker\Phi_T\cap\mathcal L_{\Delta_m}\), along which the Gramian is
unchanged. It is therefore possible that
$\Ker\Phi_T\cap\mathcal L_{C_S}
=
\{0\}$
even when
$\Ker\Phi_T\cap \mathcal L_{\Delta_m}
\neq
\{0\}$.

To obtain a computable characterization of
Condition~\ref{cond:generalized_CS_restricted_injectivity}, we introduce
reduced coordinates for the feasible-direction space
\(\mathcal L_{C_S}\) in
\eqref{eq:candidate_exclusion_direction_space}.
Using the ordering \eqref{ordering}, define
\begin{align}
E_S
:=
\begin{bmatrix}
e_{i_1}^{(m)} & \cdots & e_{i_k}^{(m)}
\end{bmatrix}
\in\Real^{m\times k},
\label{eq:candidate_selection_matrix}
\end{align}
where \(e_i^{(m)}\) denotes the \(i\)th standard basis vector of
\(\Real^m\).
The matrix \(E_S\) embeds the retained allocation coordinates into
\(\Real^m\) by placing them at the indices in \(S\) and setting all
excluded coordinates to zero.
This reduced-coordinate representation converts
Condition~\ref{cond:generalized_CS_restricted_injectivity} into a
matrix-rank condition and expresses the feasible-direction separation
constant as a minimum singular value.

The following lemma collects the corresponding reduced-coordinate
identities.

\begin{lemma}
\label{lem:candidate_exclusion_reduced_coordinates}
The following statements hold.

\begin{enumerate}
\item
The map $q\mapsto E_Sq$ is a bijection from $\Delta_k$ onto $C_S$.
In particular, every $p\in C_S$ has the unique representation
\begin{align}
p=E_Sq,
\qquad
q\in\Delta_k.
\label{eq:reduced_coordinate_parameterization}
\end{align}

\item
For $q\in\Delta_k$, the Gramian map $\Phi_T$ is
\begin{align}
\Phi_T(E_Sq)
=
\sum_{j=1}^k q_jW_{i_j}(T).
\label{eq:reduced_coordinate_Gramian}
\end{align}

\item
The Euclidean projection onto $C_S$ satisfies
\begin{align}
P_{C_S}(z)
=
E_SP_{\Delta_k}\left(E_S^\top z\right)
\label{eq:projection_onto_candidate_face}
\end{align}
for every $z\in\Real^m$.

\item
The feasible-direction space is
\begin{align}
\mathcal L_{C_S}
=
\left\{
E_Sz
\ \middle|\
z\in\Real^k,\ \bm 1^\top z=0
\right\}.
\label{eq:reduced_feasible_direction_space}
\end{align}
\end{enumerate}
\end{lemma}

\begin{proof}
See Appendix~\ref{app:feasibility_and_candidate_exclusion}. \qed
\end{proof}

Lemma~\ref{lem:candidate_exclusion_reduced_coordinates} shows that the
candidate-exclusion problem \eqref{prob:generalized_controllability_scoring} with $C=C_S$ can be solved directly over the
lower-dimensional simplex $\Delta_k$, without retaining the coordinates
fixed at zero. Algorithm~\ref{alg:PG_general} therefore requires only a
simplex projection in $\Real^k$, followed by the embedding through $E_S$.
All excluded coordinates remain identically zero. The projection
$P_{\Delta_k}$ can be computed exactly in finite time by standard
dedicated algorithms \cite{condat2016simplex}.

The same reduction also provides a convenient starting point for testing
Condition~\ref{cond:generalized_CS_restricted_injectivity}. Under the identification of \(C_S\) with \(\Delta_k\) through
\(p=E_Sq\), the feasible-direction space
\(\mathcal L_{C_S}\) is identified with the zero-sum subspace
\begin{align}
\left\{
z\in\Real^k
\ \middle|\
\bm 1^\top z=0
\right\}. \label{z_subspace}
\end{align}
This subspace still carries one linear equality constraint. To eliminate
that constraint and represent every feasible direction by an unconstrained
coordinate vector, we introduce an orthonormal basis of the zero-sum
subspace.

If \(k=1\), then \(C_S\) is a singleton and
\(\mathcal L_{C_S}=\{0\}\), so no restricted-injectivity or curvature test
is required. Suppose henceforth that \(k\geq2\), and let
the columns of
\(Z_S\in\Real^{k\times(k-1)}\) form an orthonormal basis of the
zero-sum subspace \eqref{z_subspace}.
Equivalently,
\begin{align}
Z_S^\top Z_S
&=
I_{k-1},
&
\operatorname{Im}Z_S
&=
\left\{
z\in\Real^k
\ \middle|\
\bm 1^\top z=0
\right\}.
\label{eq:ZS_zero_sum_image}
\end{align}
Every vector \(z\in\Real^k\) in \eqref{z_subspace} then has a unique representation
\begin{align}
z=Z_Sy,
\qquad
y\in\Real^{k-1}. \label{unique_z}
\end{align}
Substituting this representation
into \eqref{eq:reduced_feasible_direction_space} gives
\begin{align}
\mathcal L_{C_S}
=
\left\{
E_SZ_Sy
\ \middle|\
y\in\Real^{k-1}
\right\}.
\label{eq:reduced_feasible_direction_parameterization}
\end{align}
Thus, every feasible direction \(v\in\mathcal L_{C_S}\) has the unique
representation
\begin{align}
v=E_SZ_Sy,
\qquad
y\in\Real^{k-1}. \label{unique_v}
\end{align}
Moreover, since both \(E_S\) and \(Z_S\) have orthonormal columns,
\begin{align}
\|E_SZ_Sy\|
=
\|y\|
\label{eq:reduced_feasible_direction_isometry}
\end{align}
for every \(y\in\Real^{k-1}\).

To express the Gramian change generated by a feasible allocation direction
as a matrix-vector product, let \(\operatorname{vec}\) denote column-wise
vectorization and define
\begin{align}
M_S
:=
\begin{bmatrix}
\operatorname{vec}\left(W_{i_1}(T)\right)
&
\cdots
&
\operatorname{vec}\left(W_{i_k}(T)\right)
\end{bmatrix}
\in\Real^{n^2\times k}.
\label{eq:restricted_Gramian_matrix}
\end{align}
For any \(z\in\Real^k\), linearity of the Gramian map \(\Phi_T\) gives
\begin{align}
\operatorname{vec}\left(\Phi_T(E_Sz)\right)
=
\operatorname{vec}\left(
\sum_{j=1}^k z_jW_{i_j}(T)
\right)=
M_Sz.
\label{eq:Gramian_map_retained_coordinates}
\end{align}
Thus, \(M_S\) represents the Gramian map with respect to perturbations of
the retained allocation coordinates. Such a perturbation is feasible for
the simplex face only when \(\bm 1^\top z=0\). Since every such vector has
the unique representation
\eqref{unique_z},
every feasible direction can be written as
\eqref{unique_v}, and
\eqref{eq:Gramian_map_retained_coordinates} becomes
\begin{align}
\operatorname{vec}\left(\Phi_T(E_SZ_Sy)\right)
=
M_SZ_Sy.
\label{eq:Gramian_map_reduced_coordinates}
\end{align}
Therefore, \(M_SZ_S\) represents the restriction of the Gramian map to the
feasible-direction space \(\mathcal L_{C_S}\), expressed in the
unconstrained coordinates \(y\in\Real^{k-1}\). In particular, Condition~\ref{cond:generalized_CS_restricted_injectivity} is equivalent to
\begin{align}
M_SZ_Sy=0
\quad\Rightarrow\quad
y=0, \label{Condition2}
\end{align}
that is, to \(M_SZ_S\) having full column rank. Moreover, the isometry
\eqref{eq:reduced_feasible_direction_isometry} allows the
feasible-direction separation constant to be expressed as the minimum
singular value of \(M_SZ_S\). The precise characterization is stated next.

\begin{proposition}
\label{prop:computable_restricted_injectivity}
Suppose that \(k\geq2\). Then the following statements hold.

\begin{enumerate}
\item
For every \(y\in\Real^{k-1}\),
\begin{align}
\left\|
\Phi_T(E_SZ_Sy)
\right\|_{\rm F}
=
\|M_SZ_Sy\|.
\label{eq:restricted_injectivity_matrix_test}
\end{align}

\item
The Gramian map \(\Phi_T\) is injective on
\(\mathcal L_{C_S}\) if and only if \(M_SZ_S\) has full column rank.
Equivalently,
\begin{align}
\sigma_{\min}(M_SZ_S)>0.
\label{eq:restricted_injectivity_singular_value_test}
\end{align}

\item
For \(C=C_S\), the constant $\kappa_{C_S,T}$ in \eqref{eq:explicit_curvature_parameters1} satisfies
\begin{align}
\kappa_{C_S,T}
=
\sigma_{\min}(M_SZ_S).
\label{eq:kappa_as_minimum_singular_value}
\end{align}

\item
For \(C=C_S\), the constant $\beta_T$ in \eqref{eq:explicit_curvature_parameters2} satisfies
\begin{align}
\beta_T
=
\max_{i\in S}
\lambda_{\max}\left(W_i(T)\right).
\label{eq:beta_on_simplex_face}
\end{align}
\end{enumerate}
\end{proposition}

\begin{proof}
See Appendix~\ref{app:restricted_injectivity_and_linear_convergence}. \qed
\end{proof}

Proposition~\ref{prop:computable_restricted_injectivity} provides both a
binary test and a quantitative measure of restricted injectivity.
Full column rank of \(M_SZ_S\) determines whether Condition~\ref{cond:generalized_CS_restricted_injectivity}
holds, whereas \(\sigma_{\min}(M_SZ_S)\) is the minimum Frobenius-norm
change in the Gramian generated by a unit Euclidean-norm feasible
allocation direction. A small positive value indicates that the Gramian
map is nearly degenerate along some feasible direction, even though it
remains injective. Through
\(\kappa_{C_S,T}=\sigma_{\min}(M_SZ_S)\), this singular value enters
quadratically into the explicit strong-convexity parameters and therefore
quantifies the curvature used in the linear-convergence analysis.


\section{Numerical Experiments}
\label{sec:numerical_experiment}
This section illustrates how candidate exclusion affects feasibility,
restricted injectivity, and the optimal controllability-score allocations
in a directed network.

We consider the ten-node directed network depicted in
Fig.~\ref{fig:network}. Every edge has weight \(0.2\), and
\(A=-L\), where \(L\) is the directed graph Laplacian.
We specialize problem~\eqref{prob:generalized_controllability_scoring}
to node-wise actuation by setting
\begin{align}
m=n=10,
\qquad
b_i=e_i,
\quad
i=1,\ldots,10.
\end{align}

For each horizon \(T\) and each objective
\(f\in\{f_{\rm V},f_{\rm A}\}\), we solve two instances of
problem~\eqref{prob:generalized_controllability_scoring}.
The full problem uses \(C=\Delta_{10}\). For the restricted problem,
we retain the input locations
$S
:=
\{1,2,3,4,7,9,10\}$
and use the simplex face
\begin{align}
C_S
&:=
\left\{
p\in\Delta_{10}
\ \middle|\
p_i=0
\text{ for every }i\notin S
\right\}\notag\\
&=
\left\{
p\in\Delta_{10}
\ \middle|\
p_5=p_6=p_8=0
\right\}.
\end{align}
Nodes \(5\), \(6\), and \(8\) remain part of the network dynamics but
are unavailable as direct input locations in the restricted problem.

\begin{figure}[t]
    \centering
    \begin{tikzpicture}[
        scale=0.7,
        every node/.style={font=\normalsize},
        admissible/.style={
            circle,
            draw,
            thick,
            minimum size=0.56cm,
            inner sep=0pt
        },
        excluded/.style={
            admissible,
            very thick,
            dashed,
            fill=gray!30
        },
        edge/.style={
            -Latex,
            thick
        }
    ]
        \node[admissible] (n1)  at (4.1,2.3) {1};
        \node[admissible] (n2)  at (0.5,2.3) {2};
        \node[admissible] (n3)  at (2.9,2.3) {3};
        \node[admissible] (n4)  at (1.7,2.3) {4};
        \node[admissible] (n7)  at (2.3,3.5) {7};
        \node[admissible] (n9)  at (5.1,3.5) {9};
        \node[admissible] (n10) at (0.5,1.0) {10};

        \node[excluded] (n5) at (4.1,0.0) {5};
        \node[excluded] (n6) at (1.7,0.0) {6};
        \node[excluded] (n8) at (2.9,0.0) {8};

        \draw[edge] (n1) -- (n5);
        \draw[edge] (n2) -- (n10);
        \draw[edge] (n3) -- (n8);
        \draw[edge] (n4) -- (n6);
        \draw[edge] (n7) -- (n1);
        \draw[edge] (n7) -- (n2);
        \draw[edge] (n7) -- (n3);
        \draw[edge] (n7) -- (n4);
        \draw[edge] (n9) -- (n1);
        \draw[edge] (n10) -- (n6);

\node[admissible, minimum size=0.42cm]
    at (0.55,-0.85) {};
\node[
    anchor=west,
    font=\scriptsize
] at (0.90,-0.85)
    {intervenable};

\node[excluded, minimum size=0.42cm]
    at (3.55,-0.85) {};
\node[
    anchor=west,
    font=\scriptsize
] at (3.90,-0.85)
    {non-intervenable};
    \end{tikzpicture}
      \caption{Directed network used in the numerical experiment.
    Nodes \(5\), \(6\), and \(8\), indicated by shaded dashed circles,
    remain in the network dynamics but are excluded as candidate input
    locations.}
    \label{fig:network}
\end{figure}
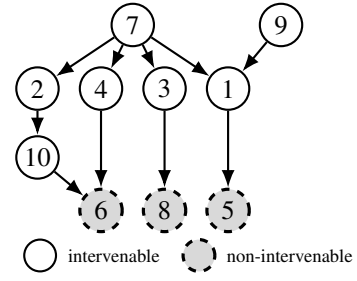

The resulting restricted input pair $(A,B_S)$ is controllable. Hence,
Corollary~\ref{cor:candidate_exclusion_feasibility} implies that
$C_S\cap X_T\neq\emptyset$
for every \(T>0\). Thus, excluding direct intervention at nodes
\(5\), \(6\), and \(8\) preserves full-state controllability and
makes the restricted scoring problem feasible at every finite horizon.

This conclusion depends essentially on retaining node \(9\). Indeed, if
node \(9\) is also excluded, the admissible set becomes
$\{1,2,3,4,7,10\}$.
The corresponding controllability matrix has rank \(9\), and the
resulting simplex face is therefore infeasible for every \(T>0\).
Because node \(9\) has no incoming edge, removing its input leaves its
source mode unactuated. Node \(9\) is thus required for feasibility
under the present exclusion pattern.

We evaluate the full and restricted problems at
$T\in\{1,10,100\}$
to examine the allocation behavior from shorter to longer horizons.
The restricted-injectivity quantities from
Proposition~\ref{prop:computable_restricted_injectivity} satisfy
\begin{align}
\left(
\sigma_{\min}(M_S(T)Z_S)
\right)_{T=1,10,100}
=
\left(
0.7200,\,
1.8310,\,
1.8232
\right).
\label{eq:numerical_restricted_injectivity_values}
\end{align}
Their positivity verifies restricted injectivity at all three horizons.
Together with feasibility, this verifies the assumptions of
Theorem~\ref{thm:generalized_CS_linear_convergence}, which guarantee
uniqueness of both the VCS and AECS optimizers and linear convergence of
the exact algorithm with \(\epsilon=0\).

For the numerical computations, Algorithm~\ref{alg:PG_general} is
initialized at the uniform allocation on the corresponding simplex or
simplex face, with
\begin{align}
\bar\alpha=1,\qquad
\sigma=0.1,\qquad
\rho=0.5,\qquad
\epsilon=10^{-4}.
\end{align}
Table~\ref{tab:directed_allocations} reports the computed allocations
for the full and restricted problems. Boldface and underlining indicate
the largest and second-largest entries in each column, respectively.

\begin{table*}[t]
\centering
\caption{Optimal finite-horizon VCS and AECS allocations at
$T=1$, $10$, and $100$ over the full simplex and the simplex face
excluding nodes \(5\), \(6\), and \(8\).}
\label{tab:directed_allocations}
{\scriptsize
\setlength{\tabcolsep}{2.3pt}
\renewcommand{\arraystretch}{1.03}

\begin{tabular}[t]{@{}c|cc|cc@{}}
\multicolumn{5}{c}{(a) \(T=1\)}\\
\hline
& \multicolumn{2}{c|}{VCS}
& \multicolumn{2}{c}{AECS}\\
Node
& full
& restricted
& full
& restricted\\
\hline
1  & 0.09967 & \underline{0.19960}
   & \textbf{0.10951} & \textbf{0.30714}\\
2  & 0.10000 & 0.09991
   & 0.09999 & 0.01748\\
3  & 0.10000 & \textbf{0.19980}
   & 0.09999 & \underline{0.30526}\\
4  & 0.09999 & 0.14975
   & 0.09998 & 0.16919\\
5  & 0.09967 & 0
   & 0.09976 & 0\\
6  & 0.09935 & 0
   & \underline{0.10927} & 0\\
7  & \textbf{0.10132} & 0.10098
   & 0.09111 & 0.01589\\
8  & 0.09967 & 0
   & 0.09979 & 0\\
9  & \underline{0.10033} & 0.10020
   & 0.09065 & 0.01585\\
10 & 0.10000 & 0.14976
   & 0.09996 & 0.16919\\
\hline
\end{tabular}
\hfill
\begin{tabular}[t]{@{}c|cc|cc@{}}
\multicolumn{5}{c}{(b) \(T=10\)}\\
\hline
& \multicolumn{2}{c|}{VCS}
& \multicolumn{2}{c}{AECS}\\
Node
& full
& restricted
& full
& restricted\\
\hline
1  & 0.08647 & 0.16166
   & \textbf{0.14966} & \textbf{0.25120}\\
2  & 0.10118 & 0.09709
   & 0.10536 & 0.07617\\
3  & 0.10247 & \textbf{0.18172}
   & 0.10745 & \underline{0.20201}\\
4  & 0.09196 & 0.13031
   & 0.10173 & 0.18620\\
5  & 0.08632 & 0
   & 0.08782 & 0\\
6  & 0.07800 & 0
   & \underline{0.13187} & 0\\
7  & \textbf{0.16835} & \underline{0.17561}
   & 0.07685 & 0.05655\\
8  & 0.07882 & 0
   & 0.08958 & 0\\
9  & \underline{0.11603} & 0.12237
   & 0.05198 & 0.04067\\
10 & 0.09041 & 0.13124
   & 0.09769 & 0.18720\\
\hline
\end{tabular}
\hfill
\begin{tabular}[t]{@{}c|cc|cc@{}}
\multicolumn{5}{c}{(c) \(T=100\)}\\
\hline
& \multicolumn{2}{c|}{VCS}
& \multicolumn{2}{c}{AECS}\\
Node
& full
& restricted
& full
& restricted\\
\hline
1  & 0.07369 & 0.09222
   & \textbf{0.16757} & 0.18555\\
2  & 0.10147 & 0.09750
   & 0.11255 & 0.09312\\
3  & 0.10887 & 0.15220
   & 0.11990 & \underline{0.18656}\\
4  & 0.08653 & 0.10836
   & 0.10462 & 0.16679\\
5  & 0.05120 & 0
   & 0.08612 & 0\\
6  & 0.06119 & 0
   & \underline{0.13301} & 0\\
7  & \textbf{0.24672} & \textbf{0.25182}
   & 0.09103 & 0.08282\\
8  & 0.04344 & 0
   & 0.06967 & 0\\
9  & \underline{0.16003} & \underline{0.19257}
   & 0.02117 & 0.09291\\
10 & 0.06687 & 0.10534
   & 0.09436 & \textbf{0.19225}\\
\hline
\end{tabular}
}
\end{table*}

\subsection{Short-horizon interpretation}

For the full problem and any fixed
\(p\in\operatorname{ri}\Delta_{10}\),
\begin{align}
W(p,T)
&=
T\operatorname{diag}(p)+\mathcal{O}(T^2),\notag\\
f_{\rm V}(p)
&=
-10\log T
-\sum_{i=1}^{10}\log p_i
+\mathcal{O}(T),
\label{eq:full_short_VCS_numerical}\\
f_{\rm A}(p)
&=
\frac{1}{T}
\sum_{i=1}^{10}\frac{1}{p_i}
+\mathcal{O}(1)
\label{eq:full_short_AECS_numerical}
\end{align}
as \(T\to0\).
Hence, the optimal full-simplex allocations converge to
\(\frac{1}{10}\bm 1_{10}\)
\cite[Lemma~4 and Theorem~4]{sato2025uniqueness}.
The computed allocations at \(T=1\) are already close to this limit.

For the restricted problem, however,
\(T\operatorname{diag}(p)\) is singular because
\(p_5=p_6=p_8=0\).
Higher-order propagation through the network therefore affects the
allocation even at short horizons.
The following discussion is only an interpretation of the computed
allocations at \(T=1\), not a rigorous \(T\to0\) asymptotic result.

The restricted VCS allocation is close to the profile assigning
\(0.2\) to nodes \(1\) and \(3\), \(0.15\) to nodes \(4\) and \(10\),
and \(0.1\) to nodes \(2\), \(7\), and \(9\).
This is consistent with nodes \(1\) and \(3\) compensating for the
excluded nodes \(5\) and \(8\), while nodes \(4\) and \(10\) provide
short paths to node \(6\).

The restricted AECS allocation is more concentrated:
\begin{align}
p_1&=0.30714,&
p_3&=0.30526,&
p_4+p_{10}&=0.33838,
\end{align}
whereas \(p_2+p_7+p_9=0.04922\).
This concentration is consistent with the stronger emphasis of AECS on
weakly controllable directions.

Thus, the numerical results suggest that candidate exclusion makes the
short-horizon allocation both network dependent and criterion dependent.

\subsection{Horizon dependence and criterion-dependent reallocation}

At \(T=10\), contributions from longer directed paths become more
pronounced. Node \(7\) directly influences nodes \(1\), \(2\), \(3\),
and \(4\), and reaches nodes \(5\), \(6\), \(8\), and \(10\) through
subsequent paths. Consistently, its full VCS allocation increases from
\(0.10132\) at \(T=1\) to \(0.16835\) at \(T=10\). In the restricted
VCS problem, this broad downstream influence is combined with local
compensation: nodes \(3\) and \(7\) receive the two largest
allocations, while nodes \(1\), \(4\), and \(10\) retain substantial
weights because they provide short paths to the excluded terminal
states. By contrast, the restricted AECS allocation remains
concentrated at nodes \(1\), \(3\), \(4\), and \(10\), which together
receive approximately \(82.7\%\) of the resource.

This difference can be interpreted from the spectral forms of the
objective functions underlying the two allocations. As a Gramian
eigenvalue approaches zero, its reciprocal in the inverse-trace
objective defining AECS grows more rapidly than its negative logarithm
in the log-determinant objective defining VCS. The larger AECS
reallocation is therefore consistent with its stronger emphasis on
the relatively weak controllability directions associated with the
excluded terminal states.

At \(T=100\), VCS places greater emphasis on the source nodes \(7\)
and \(9\). Their combined allocations are
\begin{align}
p_7+p_9
&=
0.40675
&&\text{in the full problem},\notag\\
p_7+p_9
&=
0.44439
&&\text{in the restricted problem}.
\end{align}
This pattern is consistent with the increasing long-horizon importance
of the upstream source nodes. At the same time, the total full VCS
allocation assigned to the excluded nodes decreases from \(0.29869\)
at \(T=1\) to \(0.15583\) at \(T=100\). Thus, as the horizon increases,
the unrestricted VCS itself places progressively less weight on the
terminal nodes.

AECS retains a different long-horizon pattern. Even at \(T=100\),
nodes \(1\), \(3\), \(4\), and \(10\) receive approximately \(73.1\%\)
of the restricted allocation. Hence, AECS continues to emphasize
input locations that improve the relatively weak directions associated
with the excluded terminal states, whereas VCS increasingly emphasizes
input locations with broad downstream influence.

To summarize the overall reallocation caused by exclusion, define
$d_{\rm V}(T)
:=
\left\|
p_{\rm full}^{\rm V}(T)-p_{\rm res}^{\rm V}(T)
\right\|_1$ and
$d_{\rm A}(T)
:=
\left\|
p_{\rm full}^{\rm A}(T)-p_{\rm res}^{\rm A}(T)
\right\|_1$.
The values at \(T=1\), \(10\), and \(100\) are
\begin{align}
d_{\rm V}(T)
&=
0.5985,\ 0.4945,\ 0.3196,\\
d_{\rm A}(T)
&=
1.0827,\ 0.7402,\ 0.6329,
\end{align}
respectively. Thus, exclusion induces a larger change in the optimal
AECS allocation at every reported horizon. The VCS distance decreases
markedly with \(T\), because the unrestricted VCS itself places
progressively less weight on the excluded terminal nodes, whereas the
AECS distance remains substantial even at \(T=100\).

\section{Conclusions}
\label{sec:conclusion}

\subsection{Summary}

We developed a domain-aware Armijo projected-gradient analysis for convex
objectives that are smooth only on an open domain.
Under the boundary-blow-up condition, every run initialized at a feasible
point remains in a compact invariant sublevel set separated from the
complement of the objective domain.
The resulting safe-neighborhood localization yields well-defined and
finitely terminating backtracking, a run-specific but
iteration-independent positive lower bound on the accepted step sizes,
explicit objective and stationarity bounds, and convergence of the full
iterate sequence.
Positive curvature restricted to feasible displacement directions further
guarantees uniqueness and linear convergence.

For controllability scoring with prescribed input directions and compact
convex allocation constraints, the affine dependence of the finite-horizon
Gramian on the allocation yields an exact feasibility characterization in
terms of controllability of the candidate input directions eligible for
positive allocation.
Restricted injectivity of the Gramian map on the feasible-direction space
provides explicit strong-convexity bounds and hence uniqueness and linear
convergence.
For candidate-exclusion constraints, the same condition admits a
computable reduced-coordinate characterization through a minimum singular
value.

In the directed-network example, excluding nodes \(5\), \(6\), and \(8\)
preserved both feasibility and restricted injectivity at the three reported
horizons, whereas additionally excluding node \(9\) destroyed feasibility.
The computed optimal allocations varied substantially with both the
criterion and the horizon.
At \(T=1\), \(10\), and \(100\), candidate exclusion produced larger
\(\ell_1\) allocation changes for AECS than for VCS.
In this example, the VCS allocation distance decreased as the horizon
increased, while the AECS distance remained substantial at the reported
horizons.

\subsection{Extensions and broader applicability}

The framework developed in this paper is not specific to
time-invariant controllability scoring. On a fixed time interval, the
LTV controllability Gramian considered in
\cite{umezu2026controllability} remains linear in the allocation
variables. Therefore, the feasibility, boundary-separation, and
domain-aware convergence arguments can be transferred to prescribed
LTV input directions under compact convex allocation constraints.
Likewise, by controllability--observability duality, the same framework
applies to the observability scores proposed in
\cite[Sec.~VII]{sato2024scores}. In that setting, observability of the
active sensor family replaces controllability of the active input
family, while restricted injectivity is imposed on the corresponding
observability-Gramian map.

The same matrix-allocation structure also appears in finite-candidate
approximate optimal experimental design (OED), where the information
matrix is affine in the experimental design weights
\cite{pukelsheim2006optimal,huan2024optimal}. As shown in
\cite{sato2026relationship}, this structure establishes a direct
correspondence between VCS and D-optimality and between AECS and
A-optimality. Hence, for a compact convex set of admissible design
weights, the present open-domain analysis is directly applicable when
the active experimental conditions collectively generate a
positive-definite information matrix. In this setting, restricted
injectivity requires that no nonzero feasible perturbation of the design
weights leave the information matrix unchanged. Under this condition,
the feasible-direction curvature argument gives uniqueness and linear
convergence. If restricted injectivity fails, the existence and
full-sequence convergence results still apply under the same
positive-definiteness condition, although the optimal design weights may
be nonunique.

\section*{Acknowledgments}
This work was supported by JST PRESTO, Japan, Grant Number JPMJPR25K4. 
The author used OpenAI ChatGPT throughout the preparation and revision of
this manuscript. All assumptions, definitions,
theorems, proofs, citations, and numerical results were independently
verified by the author, who takes full responsibility for the content of
the manuscript.

\appendix

\subsection{Proofs of Lemma~\ref{lem:sublevel_geometry} and
Corollary~\ref{thm:existence}}
\label{app:sublevel_geometry_and_attainment}

\begin{proof}[Lemma~\ref{lem:sublevel_geometry}]
The set $\mathcal F^{(0)}$ is nonempty because it contains
$p^{(0)}$. Applying Assumption~\ref{ass:blowup} with
$\alpha=f(p^{(0)})$ gives
\eqref{eq:sublevel_boundary_distance}.
We next prove closedness. Let
$\{p^{(j)}\}\subset\mathcal F^{(0)}$ converge to
$\bar p\in\Real^n$. Since $C$ is compact and hence closed,
$\bar p\in C$. If $U^c=\emptyset$, then $U=\Real^n$ and
$\bar p\in U$. If $U^c\neq\emptyset$, continuity of the distance function
gives
\begin{align}
\dist(p^{(j)},U^c)
\geq
\dist(\mathcal F^{(0)},U^c)>0
\end{align}
for every $j$. Passing to the limit yields
$\dist(\bar p,U^c)>0$, and hence $\bar p\in U$. Continuity of $f$ on $U$
now gives
\begin{align}
f(\bar p)
=
\lim_{j\to\infty}f(p^{(j)})
\leq
f(p^{(0)}).
\end{align}
Thus, $\bar p\in\mathcal F^{(0)}$, so $\mathcal F^{(0)}$ is closed.
Because it is a closed subset of the compact set $C$, it is compact.
Finally, $C\cap U$ is convex by Assumption~\ref{ass:set}, and a sublevel
set of the convex function $f$ is convex. Hence,
$\mathcal F^{(0)}$ is convex.
\qed
\end{proof}

\begin{proof}[Corollary~\ref{thm:existence}]
Let $p^{(0)}\in C\cap U$ be the feasible point used to define
$\mathcal F^{(0)}$. By Lemma~\ref{lem:sublevel_geometry},
the set $\mathcal F^{(0)}$ is nonempty and compact. Since $f$ is
continuous on $\mathcal F^{(0)}$, the Weierstrass theorem \cite[Prop.~A.8]{bertsekas2016nonlinear} yields a point
$p^\star\in\mathcal F^{(0)}$ that minimizes $f$ over
$\mathcal F^{(0)}$.
For every $p\in(C\cap U)\setminus\mathcal F^{(0)}$, we have
\begin{align}
f(p)>f\bigl(p^{(0)}\bigr)\geq f(p^\star).
\end{align}
Hence, $p^\star$ minimizes $f$ over all of $C\cap U$ and is therefore
an optimal solution of problem~\eqref{prob:general}.

If $f$ is strictly convex on the convex set $\mathcal F^{(0)}$, the optimal solution is unique by \cite[Prop.~1.1.2]{bertsekas2016nonlinear}.
\qed
\end{proof}

\subsection{Proof of Theorem~\ref{thm:well_defined}}
\label{app:uniform_open_domain_safeguards}

\begin{proof}[Theorem~\ref{thm:well_defined}]
We first verify that the neighborhood $\mathcal K_r$ remains inside the
objective domain. Suppose that $U^c\ne\emptyset$. For every
$q\in\mathcal K_r$, compactness of $\mathcal F^{(0)}$ gives
$p\in\mathcal F^{(0)}$ such that $\|q-p\|\leq r$. Since
$r\leq d/2$, the triangle inequality for set distance yields
\begin{align}
\dist(q,U^c)
\geq
\dist(p,U^c)-\|q-p\|
\geq
d-r
\geq
\frac{d}{2}
>0.
\label{eq:safe_neighborhood_distance}
\end{align}
Hence, $\mathcal K_r\subset U$. The same inclusion is immediate when
$U^c=\emptyset$.

We next obtain a uniform smoothness constant on this neighborhood. Since
$f$ is twice continuously differentiable on $U$ and $\mathcal K_r$ is
compact and convex, the constant $L$ defined in
\eqref{eq:uniform_acceptance_constants} is finite and is a Lipschitz
constant of $\nabla f$ on $\mathcal K_r$.

Now fix $p\in\mathcal F^{(0)}$. Since
$p\in\mathcal F^{(0)}$, the displacement bound
\eqref{eq:cited_trial_displacement} and the definition of $G$ give
\begin{align}
\dist\bigl(q_\alpha(p),\mathcal F^{(0)}\bigr)
\leq
\|q_\alpha(p)-p\|
\leq
\alpha\|\nabla f(p)\|
\leq
\alpha G.
\label{eq:trial_safe_neighborhood_bound}
\end{align}
Therefore, $\alpha\leq r/G$ implies
$q_\alpha(p)\in\mathcal K_r\subset U$.

Because the segment joining \(p\) and \(q_\alpha(p)\) lies in the convex
set \(\mathcal K_r\), the descent lemma
\cite[Lem.~5.7]{beck2017first} gives
\begin{align}
f(q_\alpha(p))
&\leq
f(p)
+\nabla f(p)^\top(q_\alpha(p)-p)
+\frac{L}{2}\|q_\alpha(p)-p\|^2.
\notag
\end{align}
Combining this inequality with
\eqref{eq:cited_projection_descent} yields
\begin{align}
f(q_\alpha(p))
&\leq
f(p)
+\left(1-\frac{L\alpha}{2}\right)
\nabla f(p)^\top(q_\alpha(p)-p).
\label{eq:armijo_threshold_derivation}
\end{align}
Moreover, \eqref{eq:cited_projection_descent} implies
$\nabla f(p)^\top(q_\alpha(p)-p)\leq 0$.
From $\alpha\leq\frac{2(1-\sigma)}{L}$,
\eqref{eq:armijo_threshold_derivation} therefore gives
\begin{align}
f(q_\alpha(p))
\leq
f(p)
+\sigma\nabla f(p)^\top(q_\alpha(p)-p).
\end{align}
Thus, for every \(p\in\mathcal F^{(0)}\), every trial step satisfying
$\alpha\leq\alpha_\star$
passes both the domain test and the Armijo condition \eqref{eq:armijo_acceptance_condition}.

We now apply this observation at each iteration. Suppose that
\(p^{(k)}\in\mathcal F^{(0)}\). Since
\(\bar\alpha\rho^j\to0\) as \(j\to\infty\), there exists a smallest
nonnegative integer \(j_\star\) such that
$\bar\alpha\rho^{j_\star}\leq\alpha_\star$.
The trial with step size \(\bar\alpha\rho^{j_\star}\) passes both
acceptance tests, so the backtracking loop terminates no later than this
trial. If \(j_\star\geq1\), the minimality of \(j_\star\) gives
$\bar\alpha\rho^{j_\star-1}
>
\alpha_\star$,
and hence
$\bar\alpha\rho^{j_\star}
>
\rho\alpha_\star$.
If \(j_\star=0\), then \(\bar\alpha\leq\alpha_\star\) and the first trial
is accepted. Consequently, the accepted step size satisfies
\eqref{eq:uniform_stepsize_lower_bound}.

For the accepted point
$\tilde p^{(k)}
=
q_{\alpha_k}(p^{(k)})$, the Armijo condition \eqref{eq:armijo_acceptance_condition} and \eqref{eq:cited_projection_descent} give
\begin{align}
f(\tilde p^{(k)})
\leq
f(p^{(k)})
-\frac{\sigma}{\alpha_k}
\|\tilde p^{(k)}-p^{(k)}\|^2,
\end{align}
and
\eqref{eq:projected_gradient_mapping}
yields
\eqref{eq:sufficient_decrease_mapping}.
It remains to verify invariance of the initial sublevel set. By
definition, \(p^{(0)}\in\mathcal F^{(0)}\). Suppose inductively that
\(p^{(k)}\in\mathcal F^{(0)}\). Then,
\eqref{eq:sufficient_decrease_mapping} gives
$f(\tilde p^{(k)})
\leq
f(p^{(k)})
\leq
f(p^{(0)})$.
Therefore,
$\tilde p^{(k)}\in\mathcal F^{(0)}$.
If the algorithm does not terminate at iteration \(k\), then
\(p^{(k+1)}=\tilde p^{(k)}\in\mathcal F^{(0)}\). Hence, every accepted
trial point and every generated iterate remain in
\(\mathcal F^{(0)}\).
\qed
\end{proof}

\subsection{Proofs of Theorem~\ref{thm:global_convergence} and
Corollary~\ref{cor:full_sequence_convergence}}
\label{app:localized_rates_and_full_sequence}

\begin{proof}[Theorem~\ref{thm:global_convergence}]
Fix an integer $K\geq0$ such that
$p^{(0)},\ldots,p^{(K+1)}$ are generated.
Then iterations $k=0,\ldots,K$ are nonterminal and satisfy
\begin{align}
p^{(k+1)}
=
\tilde p^{(k)}
=
P_C\left(p^{(k)}-\alpha_k\nabla f(p^{(k)})\right).
\label{eq:accepted_nonterminal_update}
\end{align}
 Theorem~\ref{thm:well_defined} supplies 
the explicit uniform bounds
$\underline\alpha\leq\alpha_k\leq\bar\alpha$.

The stationarity estimate follows by summing the sufficient-decrease
inequalities, as in \cite[Thm.~10.15]{beck2017first}.
In fact, summing \eqref{eq:sufficient_decrease_mapping} and using
$\alpha_k\geq\underline\alpha$ gives
\begin{align}
\sigma\underline\alpha
\sum_{k=0}^{K}
\|\mathcal G_{\alpha_k}(p^{(k)})\|^2
&\leq
f(p^{(0)})-f(p^{(K+1)})
\notag\\
&\leq
f(p^{(0)})-f^\star.
\label{eq:summed_sufficient_decrease}
\end{align}
The bound \eqref{eq:global_stationarity_complexity} follows by comparing
the smallest summand with the average. If the generated sequence is infinite, letting \(K\to\infty\) proves
\eqref{eq:gradient_mapping_summability}. Since the summands are
nonnegative, \eqref{eq:gradient_mapping_vanishes} follows.

The objective estimate
uses a distance-based telescoping argument related to
\cite[Thm.~10.21]{beck2017first}.
In fact, to prove \eqref{eq:global_objective_complexity}, define
$g_k
:=
\nabla f(p^{(k)})$ and
$
G_k
:=
\mathcal G_{\alpha_k}(p^{(k)})$.
Then
\begin{align}
p^{(k+1)}
=
p^{(k)}-\alpha_kG_k.
\label{eq:accepted_gradient_mapping_update}
\end{align}
Firm nonexpansiveness of $P_C$, applied to
$p^{(k)}-\alpha_k g_k$ and $p^\star=P_C(p^\star)$, yields
\begin{align}
\|p^{(k+1)}-p^\star\|^2
&\leq
\|p^{(k)}-p^\star\|^2
-2\alpha_k g_k^\top(p^{(k)}-p^\star)
\notag\\
&\quad
-\alpha_k^2\|G_k\|^2
+2\alpha_k^2 g_k^\top G_k.
\label{eq:firm_projection_objective_step}
\end{align}
By convexity of $f$,
\begin{align}
g_k^\top(p^{(k)}-p^\star)
\geq
f(p^{(k)})-f^\star.
\label{eq:convexity_objective_gap}
\end{align}
Moreover, the accepted Armijo condition
\eqref{eq:armijo_acceptance_condition} and
\eqref{eq:accepted_gradient_mapping_update} give
\begin{align}
\alpha_k g_k^\top G_k
\leq
\frac{
f(p^{(k)})-f(p^{(k+1)})
}{\sigma}.
\label{eq:armijo_gradient_inner_product}
\end{align}
Using these estimates in
\eqref{eq:firm_projection_objective_step}, dropping the nonpositive term
$-\alpha_k^2\|G_k\|^2$, and using
$\alpha_k\leq\bar\alpha$, we obtain
\begin{align}
2\alpha_k\bigl(f(p^{(k)})-f^\star\bigr)
&\leq
\|p^{(k)}-p^\star\|^2
-\|p^{(k+1)}-p^\star\|^2
\notag\\
&\quad+
\frac{2\bar\alpha}{\sigma}
\bigl(f(p^{(k)})-f(p^{(k+1)})\bigr).
\label{eq:one_step_objective_telescope}
\end{align}
Summing \eqref{eq:one_step_objective_telescope} from $k=0$ to $K$,
and using the monotonicity of the objective values together with
$\alpha_k\geq\underline\alpha$, gives
\begin{align}
& 2\underline\alpha(K+1)
\left(f(p^{(K)})-f^\star\right)
\leq
2\sum_{k=0}^{K}
\alpha_k\left(f(p^{(k)})-f^\star\right) 
\notag\\
&\leq
\|p^{(0)}-p^\star\|^2
-\|p^{(K+1)}-p^\star\|^2
\notag\\
&\quad+
\frac{2\bar\alpha}{\sigma}
\left(f(p^{(0)})-f(p^{(K+1)})\right)
\notag\\
&\leq
\|p^{(0)}-p^\star\|^2
+
\frac{2\bar\alpha}{\sigma}
\left(f(p^{(0)})-f^\star\right).
\end{align}
Dividing by $2\underline\alpha(K+1)$ proves
\eqref{eq:global_objective_complexity}.
\qed
\end{proof}

\begin{proof}[Corollary~\ref{cor:full_sequence_convergence}]
If the algorithm terminates finitely, then
$\mathcal G_{\alpha_k}(p^{(k)})=0$, so $p^{(k)}$ is optimal by
\eqref{eq:cited_gradient_mapping_optimality}. Suppose instead that the
sequence is infinite. By Theorem~\ref{thm:well_defined}, it is contained in
the compact set $\mathcal F^{(0)}$ and therefore has a cluster point.

Let \(\widehat p\) be an arbitrary cluster point of the sequence, and
choose a subsequence \(\{p^{(k_j)}\}\) such that
$p^{(k_j)}\to\widehat p$. Since
\(\alpha_{k_j}\in[\underline\alpha,\bar\alpha]\), compactness of this
interval allows us to pass to a further subsequence, not relabeled, such
that
$\alpha_{k_j}\to\widehat\alpha
\in[\underline\alpha,\bar\alpha]$.
Since \(\widehat p\in\mathcal F^{(0)}\subset C\cap U\) and
\(\widehat\alpha\geq\underline\alpha>0\), the mapping
\begin{align}
(p,\alpha)
\longmapsto
\mathcal G_\alpha(p)
=
\alpha^{-1}
\left(
p-P_C\left(p-\alpha\nabla f(p)\right)
\right)
\end{align}
is continuous at \((\widehat p,\widehat\alpha)\).
Therefore,
$\mathcal G_{\alpha_{k_j}}(p^{(k_j)})
\rightarrow
\mathcal G_{\widehat\alpha}(\widehat p)$.
On the other hand,
\eqref{eq:gradient_mapping_vanishes} gives
$\mathcal G_{\alpha_{k_j}}(p^{(k_j)})\longrightarrow 0$.
Hence,
$\mathcal G_{\widehat\alpha}(\widehat p)=0$.
By \eqref{eq:cited_gradient_mapping_optimality},
\(\widehat p\) is optimal.

It remains to show that one optimal cluster point attracts the entire
sequence. For any optimal solution $z$, inequality
\eqref{eq:one_step_objective_telescope} and nonnegativity of the objective
gap imply
\begin{align}
\|p^{(k+1)}-z\|^2
\leq
\|p^{(k)}-z\|^2+\eta_k,
\label{eq:quasi_fejer_estimate}
\end{align}
where
$\eta_k
:=
\frac{2\bar\alpha}{\sigma}
\left(f(p^{(k)})-f(p^{(k+1)})\right)
\geq 0$.
The objective values are nonincreasing and bounded below by $f^\star$;
hence
\begin{align}
\sum_{k=0}^{\infty}\eta_k
\leq
\frac{2\bar\alpha}{\sigma}
\left(f(p^{(0)})-f^\star\right)
<\infty.
\end{align}
Apply \eqref{eq:quasi_fejer_estimate} with
$z=\widehat p$, where $\widehat p$ is the optimal cluster point chosen
above. The sequence
$\|p^{(k)}-\widehat p\|^2
+\sum_{j=k}^{\infty}\eta_j$
is nonincreasing and bounded below, so
$\|p^{(k)}-\widehat p\|^2$ has a limit. Along the convergent subsequence
$p^{(k_j)}\to\widehat p$, this limit is zero. Therefore,
$p^{(k)}\to\widehat p$, which proves full-sequence convergence. \qed
\end{proof}

\subsection{Proofs of Lemma~\ref{lem:compact_strong_convexity} and
Theorem~\ref{thm:linear_convergence}}
\label{app:feasible_direction_curvature_and_linear_convergence}

\begin{proof}[Lemma~\ref{lem:compact_strong_convexity}]
Since \(\mathcal L_C\) is a finite-dimensional linear subspace of
\(\Real^n\), its unit sphere
$\mathbb S_{\mathcal L_C}
:=
\left\{
v\in\mathcal L_C
\mid
\|v\|=1
\right\}$
is compact. Hence,
\(\mathcal F^{(0)}\times\mathbb S_{\mathcal L_C}\) is compact.
Define
\begin{align}
\psi(p,v)
:=
v^\top\nabla^2 f(p)v,
\qquad
(p,v)\in
\mathcal F^{(0)}\times\mathbb S_{\mathcal L_C}.
\end{align}
Because \(f\) is twice continuously differentiable, the mapping
\(\psi\) is continuous. It therefore attains its minimum on the compact
set
\(\mathcal F^{(0)}\times\mathbb S_{\mathcal L_C}\). Thus, there exist
\(\bar p\in\mathcal F^{(0)}\) and
\(\bar v\in\mathbb S_{\mathcal L_C}\) such that
$\mu
=
\bar v^\top\nabla^2 f(\bar p)\bar v$.
By the assumption
\eqref{eq:PD_tangent}, the right-hand side is strictly positive.
Consequently,
$\mu>0$.
For every \(p\in\mathcal F^{(0)}\) and every nonzero
\(v\in\mathcal L_C\), applying the definition of \(\mu\) to
\(v/\|v\|\) gives
\begin{align}
\frac{v^\top\nabla^2 f(p)v}{\|v\|^2}
=
\left(\frac{v}{\|v\|}\right)^\top
\nabla^2 f(p)
\left(\frac{v}{\|v\|}\right)
\geq\mu.
\end{align}
Therefore, \eqref{eq:uniform_feasible_direction_curvature} holds.
The same inequality is immediate for \(v=0\).

To derive \eqref{eq:general_strong_convexity}, fix
\(p,q\in\mathcal F^{(0)}\) and set \(v:=q-p\).
Since \(\mathcal F^{(0)}\) is convex, the entire segment
\(p+tv\), \(t\in[0,1]\), lies in \(\mathcal F^{(0)}\).
Moreover, \(v\in C-C\subset\mathcal L_C\). Taylor's formula with
integral remainder and
\eqref{eq:uniform_feasible_direction_curvature} therefore give
\begin{align}
f(q)
&=
f(p)+\nabla f(p)^\top v
+\int_0^1(1-t)
v^\top\nabla^2 f(p+tv)v\,\D t
\notag\\
&\geq
f(p)+\nabla f(p)^\top(q-p)
+\frac{\mu}{2}\|q-p\|^2,
\end{align}
which proves \eqref{eq:general_strong_convexity}.
\qed
\end{proof}

\begin{proof}[Theorem~\ref{thm:linear_convergence}]
If $\mathcal L_C=\{0\}$, then $C$ is a singleton and the algorithm
terminates. We henceforth assume $\mathcal L_C\ne\{0\}$.
Lemma~\ref{lem:compact_strong_convexity} gives uniqueness. Put
$\delta_k:=f(p^{(k)})-f(p^\star)$,
$d^{(k)}:=p^{(k+1)}-p^{(k)}$, and
$t_k:=\min\{1,\mu\alpha_k\}$. 
By the minimizing property of the Euclidean projection,
\(p^{(k+1)}\) minimizes the quadratic model
$\nabla f(p^{(k)})^\top(x-p^{(k)})
+
\frac{1}{2\alpha_k}\|x-p^{(k)}\|^2$
over \(C\). Therefore, for every \(z\in C\),
\begin{align}
\nabla f(p^{(k)})^\top d^{(k)}
+\frac{\|d^{(k)}\|^2}{2\alpha_k}
&\leq
\nabla f(p^{(k)})^\top(z-p^{(k)})
\notag\\
&\quad+
\frac{\|z-p^{(k)}\|^2}{2\alpha_k}.
\label{eq:projection_model_inequality}
\end{align}
 Choose
$z=(1-t_k)p^{(k)}+t_kp^\star\in\mathcal F^{(0)}$. Strong convexity \eqref{eq:general_strong_convexity} gives
$\nabla f(p^{(k)})^\top(p^\star-p^{(k)})
\leq-\delta_k-\frac{\mu}{2}\|p^\star-p^{(k)}\|^2$,
and thus \eqref{eq:projection_model_inequality} implies that
\begin{align}
\nabla f(p^{(k)})^\top d^{(k)}
&\leq-t_k\delta_k
\nonumber\\
&\quad+\left(\frac{t_k^2}{2\alpha_k}
-\frac{\mu t_k}{2}\right)\|p^\star-p^{(k)}\|^2
\leq-t_k\delta_k,
\end{align}
because $t_k\leq\mu\alpha_k$. 
The accepted Armijo condition
\eqref{eq:armijo_acceptance_condition}, together with
$p^{(k+1)}=\tilde p^{(k)}$ and
$d^{(k)}=p^{(k+1)}-p^{(k)}$, therefore gives
\begin{align}
\delta_{k+1}
\leq
\delta_k
+
\sigma\nabla f(p^{(k)})^\top d^{(k)}
\leq
(1-\sigma t_k)\delta_k
\leq
\gamma\delta_k,
\end{align}
where
$t_k\geq\min\{1,\mu\underline\alpha\}$ follows from
$\alpha_k\geq\underline\alpha$.
 Iteration
proves the objective estimate. Finally, optimality of $p^\star$ and
\eqref{eq:general_strong_convexity} imply
$\delta_k\geq(\mu/2)\|p^{(k)}-p^\star\|^2$, which proves \eqref{eq:linear_iterate_convergence}.
\qed
\end{proof}

\subsection{Proofs of Lemma~\ref{lem:generalized_CS_differential_formulas}, Lemma~\ref{lem:generalized_CS_boundary_blowup}, and Theorem~\ref{thm:generalized_CS_global_convergence}}
\label{app:controllability_scoring_global_convergence}

\begin{proof}[Lemma~\ref{lem:generalized_CS_differential_formulas}]
Fix $p\in X_T$ and $v\in\Real^m$. By linearity of $\Phi_T$,
$W(p+sv,T)=W+sH$
for all sufficiently small $s$. The standard matrix identities
\begin{align}
\frac{\D}{\D s}\log\det(W+sH)
&=
\tr\left((W+sH)^{-1}H\right),\\
\frac{\D}{\D s}(W+sH)^{-1}
&=
-(W+sH)^{-1}H(W+sH)^{-1}
\end{align}
give \eqref{eq:generalized_VCS_first_derivative} and
\eqref{eq:generalized_AECS_first_derivative} at $s=0$. Differentiating once
more yields
\begin{align}
\left.\frac{\D^2}{\D s^2}f_{\rm V}(p+sv)\right|_{s=0}
&=
\tr(W^{-1}HW^{-1}H),\\
\left.\frac{\D^2}{\D s^2}f_{\rm A}(p+sv)\right|_{s=0}
&=
2\tr(W^{-1}HW^{-1}HW^{-1}),
\end{align}
which are \eqref{eq:generalized_VCS_second_derivative} and
\eqref{eq:generalized_AECS_second_derivative}.

To verify \eqref{eq:generalized_Hessian_kernel}, set
$S:=W^{-1/2}HW^{-1/2}$.
Then, $v^\top\nabla^2 f_{\rm V}(p)v=\|S\|_{\rm F}^2$ and
\begin{align}
v^\top\nabla^2 f_{\rm A}(p)v
=2
\tr(S^2W^{-1})
=2
\|SW^{-1/2}\|_{\rm F}^2.
\end{align}
Both second derivatives are therefore nonnegative, and either one vanishes
if and only if $S=O$. Since $W$ is positive definite, this is equivalent to
$H=O$, proving \eqref{eq:generalized_Hessian_kernel}.  \qed
\end{proof}

\begin{proof}[Lemma~\ref{lem:generalized_CS_boundary_blowup}]
Let $\{p^{(k)}\}\subset C\cap X_T$ satisfy
$\dist\left(p^{(k)},X_T^c\right)\to 0$,
and define
$\underline\lambda_k
:=
\lambda_{\min}\left(W(p^{(k)},T)\right)>0$.
Then $\underline\lambda_k\to0$. Indeed, otherwise a subsequence would satisfy
$\underline\lambda_k\geq\delta>0$. By compactness of $C$, a further
subsequence would converge to some $\bar p\in C$. The boundary-distance
condition would imply $\bar p\in X_T^c$, whereas continuity of
$\lambda_{\min}(W(p,T))$ would give
$\lambda_{\min}(W(\bar p,T))\geq\delta$, and hence $\bar p\in X_T$, a
contradiction.

The vanishing minimum eigenvalue forces both objectives to diverge.
Compactness of $C$ and continuity of $W(p,T)$ give a constant $\beta\geq1$
such that
$\|W(p,T)\|_2\leq\beta$
for all $p\in C$.
Consequently,
\begin{align}
f_{\rm V}(p^{(k)})
&\geq
-\log\underline\lambda_k-(n-1)\log\beta
\to+\infty,
\label{eq:VCS_boundary_blowup}\\
f_{\rm A}(p^{(k)})
&\geq
\frac{1}{\underline\lambda_k}
\to+\infty.
\label{eq:AECS_boundary_blowup}
\end{align}
The sequential characterization
\eqref{eq:sequential_boundary_blowup} proves the claim. \qed
\end{proof}

\begin{proof}[Theorem~\ref{thm:generalized_CS_global_convergence}]
Lemma~\ref{lem:generalized_CS_differential_formulas}, together with the
openness and convexity of $X_T$, verifies
Assumptions~\ref{ass:set} and \ref{ass:objective_regularity} with
$U=X_T$.
Lemma~\ref{lem:generalized_CS_boundary_blowup} verifies
Assumption~\ref{ass:blowup}.
Existence of an optimal solution therefore follows from
Corollary~\ref{thm:existence}.
Well-definedness and invariance of the initial sublevel set follow from
Theorem~\ref{thm:well_defined}, while full-sequence convergence follows
from Corollary~\ref{cor:full_sequence_convergence}.
Finally, Theorem~\ref{thm:global_convergence} gives
\eqref{eq:global_stationarity_complexity},
\eqref{eq:global_objective_complexity}, and
\eqref{eq:gradient_mapping_vanishes}.
\qed
\end{proof}

\subsection{Proofs of Lemma~\ref{lem:generalized_CS_uniform_curvature}, Theorem~\ref{thm:generalized_CS_linear_convergence}, and Proposition~\ref{prop:computable_restricted_injectivity}}
\label{app:restricted_injectivity_and_linear_convergence}

\begin{proof}[Lemma~\ref{lem:generalized_CS_uniform_curvature}]
Condition~\ref{cond:generalized_CS_restricted_injectivity} and compactness of the unit sphere in
$\mathcal L_C$ imply $\kappa_{C,T}>0$. For
\begin{align}
H:=\Phi_T(v),
\qquad
S:=W(p,T)^{-1/2}HW(p,T)^{-1/2},
\end{align}
the definition of $\beta_T$ gives
\begin{align}
\|S\|_{\rm F}
\geq
\frac{\|H\|_{\rm F}}
{\lambda_{\max}(W(p,T))}
\geq
\frac{\kappa_{C,T}}{\beta_T}\|v\|.
\label{eq:proof_S_lower_bound}
\end{align}
The VCS Hessian formula
\eqref{eq:generalized_VCS_second_derivative} therefore gives
\eqref{eq:explicit_mu_vcs}.

For the AECS objective,
$\lambda_{\max}(W(p,T))\leq\beta_T$ implies
$W(p,T)^{-1}
\succeq
\frac{1}{\beta_T}I$.
Since $H$ is symmetric, $S$ is symmetric and hence
$S^2\succeq O$. Using cyclic invariance of the trace, \eqref{eq:explicit_mu_aecs} follows from
\begin{align}
&v^\top\nabla^2f_{\rm A}(p)v \\
&=
2\tr\left(S^2W(p,T)^{-1}\right) = 2\tr\left(SW(p,T)^{-1}S\right)
\notag\\
&=
\frac{2}{\beta_T}\tr(S^2)
+2
\tr\left(
S\left(
W(p,T)^{-1}
-
\frac{1}{\beta_T}I
\right)S
\right)
\notag\\
&\geq
\frac{2}{\beta_T}\|S\|_{\rm F}^2
\geq
\frac{2\kappa_{C,T}^2}{\beta_T^3}\|v\|^2,
\end{align}
where the last inequality follows from
\eqref{eq:proof_S_lower_bound}.
Since \(C\cap X_T\) is convex and every displacement between two of its
points belongs to \(\mathcal L_C\), the same Taylor integral argument as
in the proof of Lemma~\ref{lem:compact_strong_convexity} shows that
\(f_{\rm V}\) and \(f_{\rm A}\) are strongly convex on \(C\cap X_T\)
with parameters \(\mu_{\rm V}\) and \(\mu_{\rm A}\), respectively. \qed
\end{proof}
\begin{proof}[Theorem~\ref{thm:generalized_CS_linear_convergence}]
By
Lemma~\ref{lem:generalized_CS_uniform_curvature}, the objective is
strongly convex on $C\cap X_T$ with parameter $\mu_{\rm V}$ or
$\mu_{\rm A}$. Existence follows from
Theorem~\ref{thm:generalized_CS_global_convergence}, and strong convexity
gives uniqueness. Since every initial sublevel set is contained in
$C\cap X_T$, Theorem~\ref{thm:linear_convergence} applies with the
corresponding strong-convexity parameter and gives the stated linear
convergence. \qed
\end{proof}
\begin{proof}[Proposition~\ref{prop:computable_restricted_injectivity}]
Equation~\eqref{eq:Gramian_map_reduced_coordinates} and the identity
\(\|\operatorname{vec}(H)\|=\|H\|_{\rm F}\) give
\eqref{eq:restricted_injectivity_matrix_test}.

Condition~\ref{cond:generalized_CS_restricted_injectivity}
is equivalent to \eqref{Condition2}, and hence also to
\eqref{eq:restricted_injectivity_singular_value_test}.

By the bijective isometric parameterization
\(v=E_SZ_Sy\) of \(\mathcal L_{C_S}\), together with
\eqref{eq:restricted_injectivity_matrix_test}, we have
\begin{align}
\kappa_{C_S,T}
=
\min_{\substack{
y\in\Real^{k-1}\\
\|y\|=1}}
\|M_SZ_Sy\| =
\sigma_{\min}(M_SZ_S),
\end{align}
proving \eqref{eq:kappa_as_minimum_singular_value},
where the last equality follows from the variational characterization
of the minimum singular value.

For every \(p\in C_S\),
\begin{align}
\lambda_{\max}(W(p,T))
&\leq
\sum_{i\in S}p_i\lambda_{\max}(W_i(T))
\notag\\
&\leq
\max_{i\in S}\lambda_{\max}(W_i(T)).
\end{align}
Equality is attained at a vertex associated with an index maximizing the
right-hand side, which proves \eqref{eq:beta_on_simplex_face}. \qed
\end{proof}

\subsection{Proofs of Lemma~\ref{lem:general_convex_feasibility}, Corollary~\ref{cor:candidate_exclusion_feasibility}, and Lemma~\ref{lem:candidate_exclusion_reduced_coordinates} }
\label{app:feasibility_and_candidate_exclusion}

\begin{proof}[Lemma~\ref{lem:general_convex_feasibility}]
We first show that every relative-interior point of $C$ has support
exactly $S_C$. Clearly, $p_i^\circ=0$ for every $i\notin S_C$. Conversely,
let $i\in S_C$ and choose $q\in C$ with $q_i>0$. Since \(p^\circ\in\operatorname{ri}C\), the definition of the relative
interior implies that, for sufficiently small \(\theta>0\),
$p^\circ+\theta(p^\circ-q)\in C$.
 If $p_i^\circ=0$, the $i$th component
of this point would be negative, contradicting
$C\subseteq\Delta_m$. Hence,
\begin{align}
p_i^\circ>0
\quad
\text{for all } i\in S_C.
\label{eq:relative_interior_full_support}
\end{align}

We next relate feasibility to the combined Gramian of the active candidate
channels. If $p\in C\cap X_T$, then \(p_i=0\) outside \(S_C\) and
\(0\leq p_i\leq1\), so
\begin{align}
\sum_{i\in S_C}W_i(T)
\succeq
\sum_{i\in S_C}p_iW_i(T)
=
W(p,T)
\succ O.
\end{align}
Conversely, suppose that
$\sum_{i\in S_C}W_i(T)\succ O$. By
\eqref{eq:relative_interior_full_support},
\(\delta:=\min_{i\in S_C}p_i^\circ>0\), and therefore
\begin{align}
W(p^\circ,T)
=
\sum_{i\in S_C}p_i^\circ W_i(T)
\succeq
\delta\sum_{i\in S_C}W_i(T)
\succ O.
\end{align}
Thus, $p^\circ\in C\cap X_T$.

Finally, $\sum_{i\in S_C}W_i(T)$ is the finite-horizon controllability
Gramian associated with $(A,B_{S_C})$. Its positive definiteness for
$T>0$ is equivalent to controllability of this pair, completing the
proof. \qed
\end{proof}

\begin{proof}[Corollary~\ref{cor:candidate_exclusion_feasibility}]
For $C=C_S$, we have $S_C=S$ and
$p^{\rm uni}\in\operatorname{ri}C_S$. Moreover,
$W(p^{\rm uni},T)
=
\frac{1}{k}\sum_{i\in S}W_i(T)$.
The result therefore follows directly from
Lemma~\ref{lem:general_convex_feasibility}. \qed
\end{proof}

\begin{proof}[Lemma~\ref{lem:candidate_exclusion_reduced_coordinates}]
By construction, $E_Sq$ places the components of $q$ at the indices in
$S$ and places zeros at all indices outside $S$. Therefore,
$q\mapsto E_Sq$ maps $\Delta_k$ bijectively onto $C_S$, proving the statement 1). The statement 2) follows from
\begin{align}
\Phi_T(E_S q)
=
\sum_{\ell=1}^m(E_Sq)_\ell W_\ell(T)
=
\sum_{j=1}^k q_jW_{i_j}(T).
\end{align}

To prove the statement 3), let \(z\in\Real^m\). Since the columns
of \(E_S\) are distinct standard basis vectors, 
$E_S^\top E_S=I_k$.
Consequently, \(E_SE_S^\top\) is the orthogonal projector onto
\(\operatorname{Im}E_S\), the coordinate subspace consisting of vectors
supported on \(S\). Its orthogonal complement is projected onto by
\(I-E_SE_S^\top\).
For any \(q\in\Delta_k\), we have \(E_Sq\in\operatorname{Im}E_S\), and
therefore
$(I-E_SE_S^\top)E_Sq = 0$.
Decomposing \(z-E_Sq\) into its components in
\(\operatorname{Im}E_S\) and
\((\operatorname{Im}E_S)^\perp\) gives
\begin{align}
z-E_Sq
&=
E_SE_S^\top(z-E_Sq)
+
(I-E_SE_S^\top)(z-E_Sq)
\notag\\
&=
E_S(E_S^\top z-q)
+
(I-E_SE_S^\top)z.
\label{eq:projection_face_orthogonal_decomposition}
\end{align}
The two terms on the right-hand side are orthogonal. Hence, the
Pythagorean identity yields
\begin{align}
\|z-E_Sq\|^2
&=
\|E_S(E_S^\top z-q)\|^2
+
\|(I-E_SE_S^\top)z\|^2
\notag\\
&=
\|E_S^\top z-q\|^2
+
\|(I-E_SE_S^\top)z\|^2,
\end{align}
where the last equality uses \(E_S^\top E_S=I_k\).
The second term on the right-hand side is independent of
\(q\). Thus,
\begin{align}
\underset{q\in\Delta_k}{\argmin}\,
\|z-E_Sq\|^2
&=
\underset{q\in\Delta_k}{\argmin}\,
\|E_S^\top z-q\|^2
\notag\\
&=
P_{\Delta_k}(E_S^\top z).
\end{align}
By statement~1), every point of \(C_S\) has the unique form \(E_Sq\)
with \(q\in\Delta_k\). Therefore,
\begin{align}
P_{C_S}(z)=\underset{p\in C_S}{\argmin}\,
\|z-p\|^2
=
E_S
\underset{q\in\Delta_k}{\argmin}\,
\|z-E_Sq\|^2.
\end{align}
Consequently,
we have \eqref{eq:projection_onto_candidate_face}.

Finally, since
$C_S-C_S
=
\left\{
E_S(q-r)
\ \middle|\
q,r\in\Delta_k
\right\}$,
every direction in $\mathcal L_{C_S}$ has the form $E_Sz$ with
$\bm 1^\top z=0$. Conversely, the zero-sum subspace is the linear span of
$\Delta_k-\Delta_k$, so every such $E_Sz$ belongs to
$\mathcal L_{C_S}$. This proves
\eqref{eq:reduced_feasible_direction_space}. \qed
\end{proof}

\bibliographystyle{IEEEtran}
\bibliography{references}




%




\end{document}